\documentclass[]{interact}

\usepackage{epstopdf}
\usepackage[caption=false]{subfig}

\usepackage[utf8]{inputenc}

\usepackage{amsmath}
\usepackage{amsthm}
\usepackage{enumerate}
\usepackage{mathdots}
\usepackage{amsfonts}
\usepackage{multicol}
\usepackage{amsmath}
\usepackage{amssymb}
\usepackage{fancybox}
\usepackage[usenames,dvipsnames]{color}
\usepackage{graphicx}
\usepackage{tikz}
\usepackage{arydshln}

\usepackage[numbers,sort&compress]{natbib}
\bibpunct[, ]{[}{]}{,}{n}{,}{,}
\makeatletter
\def\NAT@def@citea{\def\@citea{\NAT@separator}}
\makeatother

\theoremstyle{plain}

\theoremstyle{definition}

\theoremstyle{remark}

\newcommand{\C}{\mathbb{C}}

\newcommand{\efe}{\mathbb{F}}
\newcommand{\FF}{\mathbb{F}}
\newcommand{\F}{\mathbb{F}}

\newcommand{\la}{\lambda}

\def\rank{\mathop{\rm rank}\nolimits}
\def\nrank{\mathop{\rm nrank}\nolimits}
\newcommand{\wh}{\widehat}

\newtheorem{theo}{Theorem}[section]

\newtheorem{deff}[theo]{Definition}

\newtheorem{prop}[theo]{Proposition}

\newtheorem{lem}[theo]{Lemma}

\newtheorem{rem}[theo]{Remark}

\newtheorem{example}[theo]{Example}

\DeclareMathOperator{\diag}{diag}
\DeclareMathOperator{\rev}{rev}

\begin{document}

\articletype{ }

	\title{Block Full Rank Linearizations of Rational Matrices}

\author{
\name{Froil\'{a}n M. Dopico\textsuperscript{a}, Silvia Marcaida\textsuperscript{b}, Mar\'{i}a C. Quintana\textsuperscript{a} and  Paul~Van~Dooren\textsuperscript{c} 	\thanks{CONTACT Froil\'{a}n M. Dopico. Email: dopico@math.uc3m.es. Silvia Marcaida. Email: silvia.marcaida@ehu.eus. Mar\'{i}a C. Quintana. Email: maquinta@math.uc3m.es. Paul~Van~Dooren. Email: paul.vandooren@uclouvain.be}}
\affil{\textsuperscript{a}Departamento de Matem\'aticas, Universidad Carlos III de Madrid, Avda. Universidad 30, 28911 Legan\'es, Spain; \textsuperscript{b}Departamento de Matem\'{a}ticas,
	Universidad del Pa\'{\i}s Vasco UPV/EHU, Apdo. Correos 644, Bilbao 48080, Spain;
\textsuperscript{c}Department of Mathematical Engineering, Universit\'{e} catholique de Louvain, Avenue Georges Lema\^itre 4, B-1348 Louvain-la-Neuve, Belgium}
}

\maketitle

\begin{abstract}
 Block full rank pencils introduced in [Dopico et al., Local linearizations of rational matrices with application to rational approximations of nonlinear eigenvalue problems, Linear Algebra Appl., 2020] allow us to obtain local information about zeros that are not poles of rational matrices. In this paper we extend the structure of those block full rank pencils to construct linearizations of rational matrices that allow us to recover locally not only information about zeros but also about poles, whenever certain minimality conditions are satisfied. In addition, the notion of degree of a rational matrix will be used to determine the grade of the new block full rank linearizations as linearizations at infinity. This new family of linearizations is important as it generalizes and includes the structures appearing in most of the linearizations for rational matrices constructed in the literature. In particular,  this theory will be applied to study the structure and the properties of the linearizations in [P. Lietaert et al., Automatic rational approximation and linearization of nonlinear eigenvalue problems, submitted].
\end{abstract}


\begin{amscode}65F15, 15A18, 15A22, 15A54, 93B18, 93B20, 93B60 \end{amscode}

\section{Introduction}

A rational matrix $R(\la)$ is a matrix whose entries are quotients of polynomials in the scalar variable $\la$, i.e., rational functions. Zeros and poles are among the most interesting quantities attached to a rational matrix \cite{Kailath}, both from theoretical and applied points of view. The Smith--McMillan form of rational matrices \cite{McMi2} is the classical way to define their poles and zeros, together with their partial multiplicities. However, it is well-known that the Smith--McMillan form is not convenient for computing numerically poles and zeros \cite{vd1981}.

The problem of determining the pole and zero structure of rational matrices appears in several applications. Many classic problems in linear systems and control theory can be posed in terms of rational matrices \cite{Kailath,Rosen,Vard} and are related to the computation of their zeros and poles \cite{vd1981}. Currently, the computation of the zeros of rational matrices is playing a fundamental role in the very active area of Nonlinear Eigenvalue Problems (NLEPs) \cite{guttel-tisseur-2017}, either because they appear directly in rational eigenvalue problems (REPs) modeling real-life problems \cite{mehrmanvoss2004} or because other NLEPs are approximated by REPs \cite{Saad,guttel-tisseur-2017,nlep,automatic,lu-huang-bai-su-2015,van-beeumen-et-al-2018}. Recall that for a regular rational matrix $R(\la)$, the REP consists in computing numbers $\la_0$ (eigenvalues) and nonzero vectors $v$ (eigenvectors) such that $\la_0$ is not a pole of any of the entries of $R(\la)$ and $R(\la_0) v = 0$. This is equivalent to say that $\lambda_0$ is a zero of $R(\la)$ but not a pole.

One of the most reliable methods for computing the zeros and poles of a rational matrix  $R(\la)$ is via linearizations. This approach is based on constructing a matrix pencil $L(\la)$, i.e., a matrix polynomial of degree at most $1$, containing the pole and zero information of $R(\la)$ and then applying to $L(\la)$ backward stable eigenvalue algorithms, as \cite{moler-stewart,vd1979} for problems of moderate size, or Krylov methods adapted to the structure of $L(\la)$ in the large-scale setting \cite{dopico-pizarro,nlep}. The pencil $L(\la)$ is called a linearization of $R(\la)$. This classical approach was introduced at least in the late 1970s \cite{vd1981,VeDoka79} and has been revisited recently, mainly motivated by research on NLEPs and the reference \cite{su-bai-2011}. Thus, a first formal definition of linearization of a rational matrix was proposed in \cite{AlBe16}. A different definition was introduced in \cite{strong}, together with the first formal definition of strong linearization, i.e., a pencil that allows to recover both the finite and infinite pole and zero structure of $R(\la)$.

In addition to formal definitions, recent research on linearizations of rational matrices $R(\la)$ has produced new classes of strong linearizations that can be easily constructed from the polynomial part of $R(\la)$ and a minimal state-space realization of its strictly proper part. Also, the study of the recovery properties of these linearizations has received considerable attention. References in these lines include \cite{AlBe16-2,AlBe18,minimal,DasAl19laa,DaAl19_2,DaAl20,DoMaQu19}. Among the new classes of strong linearizations we mention the strong block minimal bases linearizations introduced in \cite[Theorem 5.11]{strong} since they form a very wide general family. For instance, they include as particular cases the Fiedler-like linearizations (modulo permutations) \cite[Section 8]{minimal}, are closely connected to other classes of linearizations developed in the literature \cite{DoMaQu19} and are valid for general rectangular rational matrices.

Despite the intense activity described in the previous paragraph, there are pencils that have been used in influential references as \cite{nlep,automatic} for solving numerically REPs that approximate NLEPs which do not satisfy the definitions of linearization of rational matrices given in \cite{AlBe16,strong}. The reason is that these definitions focus on pencils that allow to recover the complete pole and zero structure of rational matrices, while in \cite{nlep,automatic} only the eigenvalue information in a certain subset of the complex plane is necessary. This has motivated the development in \cite{local} of a new theory of linearizations of rational matrices in a local sense. These linearizations are pencils that preserve the structure of zeros and poles of the corresponding rational matrix in a particular subset of the underlying field and/or at infinity. Thus, such linearizations are useful in applications where only the zero and pole information in a particular region is needed. Apart from a new definition, a specific family of local linearizations of rational matrices is also introduced in \cite[Section 5]{local}, that are called block full rank pencils. These pencils are particular instances of local linearizations that allow to recover the information about zeros that are not poles of rational matrices.

The main contribution of this work is to extend the block full rank pencils from \cite{local} to a much larger family of pencils, that allow also to recover pole information under minimality conditions. This new family is called block full rank linearizations, where we use a name similar to that in \cite{local} for emphasizing the connection between both concepts. We remark that block full rank linearizations can also be seen as a wide nontrivial generalization of the strong block minimal bases linearizations introduced in \cite[Theorem 5.11]{strong}.
A second contribution of this paper is to apply block full rank linearizations to study the precise properties of the pencils used in \cite{automatic}.

The paper is organized as follows. Some preliminaries are presented in Section \ref{sec.prelim}. In Section \ref{sec-full-rank-lin}, we introduce block full rank linearizations at finite points and at infinity. The notion of degree of a rational matrix will be used to determine the grade of the new linearizations as linearizations at infinity. In Subsection \ref{SB_asBF}, we will see that the block full rank linearizations extend the structure of the strong block minimal bases linearizations in \cite{strong}. Finally, in Section \ref{automaticlin}, we will show how to view the linearizations in \cite{automatic} as block full rank linearizations, so that pole information can also be recovered from them, and provide sufficient conditions under which the linearizations in \cite{automatic} are minimal. Block full rank linearizations allow also to determine the precise properties of other linearizations appearing in the literature as those in \cite{nlep,Saad,su-bai-2011}. Conclusions are presented in Section \ref{sect:con}.

\section{Preliminaries} \label{sec.prelim}

Let $\F$ be an algebraically closed field that does not include infinity. $\efe[\la]$ denotes the ring of polynomials with coefficients in $\efe,$ and $\efe(\la)$ the field of rational functions over $\F[\la]$. $\efe^{p\times m}$, $\efe[\la]^{p\times m}$ and $\efe(\la)^{p\times m}$ denote the sets of $p\times m$ matrices with elements in $\efe,$ $\efe[\la]$ and $\efe(\la),$  respectively. The elements of $\efe[\la]^{p\times m}$ are called polynomial matrices  or matrix polynomials. A unimodular matrix is a square polynomial matrix with polynomial inverse or, equivalently,  a square polynomial matrix with nonzero constant determinant. Moreover, the elements of $\efe(\la)^{p\times m}$ are called rational matrices.

Let $R(\la)\in\F(\la)^{p\times m}$ and let $\Omega$ be a nonempty subset of $\F$. $R(\la)$ is \textit{defined or bounded in $\Omega$} if $R(\la_0)\in\F^{p\times m}$ for all $\la_0\in\Omega$ (by assuming the entries of $R(\la)$ to be irreducible rational functions). Moreover, $R(\la)\in\F(\la)^{m\times m}$ is \textit{invertible in $\Omega$} if it is defined in $\Omega$ and $\det R(\la_0)\neq 0$ for all $\la_0\in\Omega.$ $R(\la)$ is said to be \textit{defined (invertible) at $\la_0$} if it is defined (invertible) in $\Omega:=\{\la_0\}.$ A rational matrix $R(\la)\in\F(\la)^{m\times m}$ is said to be \textit{regular} if it is invertible at some $\la_0\in\F,$ that is, $\det R(\la)\not\equiv 0.$ Let $\la_0\in\F.$ Two rational matrices $G(\la), H(\la)\in\F(\lambda)^{p\times m}$ are \textit{equivalent at} $\la_0$ if there exist rational matrices $R_1(\la)\in\F(\lambda)^{p\times p}$ and $R_2(\la)\in\F(\lambda)^{m\times m}$ both invertible at $\la_0$ such that $R_1(\la)G(\la)R_2(\la)=H(\la).$ If $G(\la)$ and $H(\la)$ are equivalent for all $\la_0$ in a nonempty set $\Omega$ then $G(\la)$ and $H(\la)$ are said to be \textit{equivalent in} $\Omega$ \cite[Proposition 2.4]{local}. Let $G(\la)\in\F(\la)^{p\times m}$ and $\lambda_{0}\in\efe.$ Then $G(\la)$ is equivalent at $\la_0$ to a matrix of the form (see \cite{ AmMaZa15, vandooren-laurent-1979})
\begin{equation}\label{localsm}
\left[\begin{array}{cc}
\diag\left((\la-\la_{0})^{\nu_{1}},\ldots, (\la-\la_{0})^{\nu_{r}}\right)&0 \\
0& 0_{(p-r)\times (m-r)}
\end{array}\right],
\end{equation}
where $r$ is the normal rank of $G(\la)$, which will be denoted by $\nrank G(\la)$, and $\nu_{1}\leq\cdots\leq\nu_{r}$ are integers. The integers $\nu_{1},\ldots,\nu_{r}$ are uniquely determined by $G(\la)$ and $\la_{0}$, and are called the \textit{invariant orders at} $\la_0$ of $G(\la)$. The matrix in \eqref{localsm} is called the \textit{local Smith--McMillan form of $G(\la)$ at} $\la_{0}.$ In order to define zeros and poles we need to distinguish between positive and negative invariant orders \cite{Kailath,Vard}. Let
$\nu_1\leq\cdots\leq \nu_k<0=\nu_{k+1}=\cdots=\nu_{u-1}<\nu_u\leq\cdots\leq \nu_r$ be the invariant orders at $\la_{0}$ of $G(\la)$. Then $\la_{0}$ is said to be a \textit{pole} of $G(\la)$ with \textit{partial multiplicities} $-\nu_k,\cdots, -\nu_1,$ and a \textit{zero} of $G(\la)$ with \textit{partial multiplicities} $\nu_u,\cdots, \nu_r.$ In particular, the integers $-\nu_k,\cdots, -\nu_1$ and $\nu_u,\cdots, \nu_r$ are called the pole and zero partial multiplicities of $G(\la)$ at $\la_0,$ respectively. Moreover, $(\la-\la_0)^{-\nu_i}$ for $i=1,\ldots,k$ are called  the \textit{pole elementary divisors of $G(\la)$ at} $\la_0$, while $(\la-\la_0)^{\nu_i}$ for $i=u,\ldots,r$ are called the \textit{zero elementary divisors of $G(\la)$ at} $\la_0.$ If $G(\la)$ is a polynomial matrix then the polynomials $(\la-\la_0)^{\nu_i}$ with $\nu_i\neq 0$ are simply called \textit{elementary divisors of $G(\la)$ at} $\la_0,$ and the nonzero integers $\nu_i\neq 0$ are all positive and are called \textit{partial multiplicities of $G(\la)$ at} $\la_0.$ The (pole/zero) elementary divisors of $G(\la)$ in a nonempty subset $\Omega$ of $\F$ are the (pole/zero) elementary divisors of $G(\la)$ at every $\la_0\in\Omega$.

A matrix polynomial of the form
\begin{equation}\label{eq:polsysmat}
P(\la)=\begin{bmatrix}
A(\la) & B(\la)\\
-C(\la) & D(\la)
\end{bmatrix}
\end{equation} with $A(\la)\in\FF[\la]^ {n\times n}$ regular,   $B(\la)\in\FF[\la]^{n\times m}$, $C(\la)\in\FF[\la]^{p\times n},$ and $D(\la)\in\FF[\la]^{p\times m},$ is called a \textit{polynomial system matrix}, and the rational matrix $$G(\la)=D(\la)+C(\la)A( \la)^{-1}B(\la)$$ is called the \textit{transfer function matrix} of $P(\la)$ \cite{Rosen}. The matrix $A(\la)$ is called the \textit{state matrix} of $P(\la).$ We allow $n$ to be 0, that is, we consider $P(\la)=D(\la)$ as a polynomial system matrix with  $A(\la), B(\la)$ and $C(\la)$ as empty matrices. In this extreme case we say that $P(\la)=D(\la)$ is a polynomial system matrix with empty state matrix. Its transfer function is $D(\la)$. Let $\Omega$ be a nonempty subset of $\F.$ The polynomial system matrix $P(\la)$ in \eqref{eq:polsysmat}, with $n>0,$ is said to be \textit{minimal in} $\Omega$ if
$$\rank\begin{bmatrix} A(\la_0) \\ C(\la_0)\end{bmatrix}=\rank\begin{bmatrix} A(\la_0) & B(\la_0) \end{bmatrix}=n,$$
for all $\la_0\in\Omega.$ In the particular case that $\Omega=\{\la_0\},$ $P(\la)$ is said to be \textit{minimal at} $\la_0$ if $P(\la)$ is minimal in $\Omega.$ If $n=0$ in \eqref{eq:polsysmat}, we adopt the agreement that $P(\la)$ is minimal at every point $\la_{0}\in\F.$ We can now state the definition of linearization of a rational matrix in a subset of $\F.$ Then, in Theorem \ref{theo:spectral}, a spectral characterization is given.
\begin{deff}\cite[Definition 4.1]{local}\label{def_pointstronglin}
	Let $G(\la) \in\F(\la)^{p\times m}$ and let $\Omega\subseteq\efe$ be nonempty. Let
	\begin{equation}\label{eq_lin}
	\mathcal{L}(\la)=\left[\begin{array}{cc}
	A_1 \la +A_0 &B_1 \la +B_0\\-(C_1 \la +C_0)&D_1 \la +D_0
	\end{array}\right]\in\F[\la]^{(n+q)\times (n+r)}
	\end{equation} be a linear polynomial system matrix, with state matrix $A_1 \la +A_0$, and let
	$$\wh{G}(\la)=(D_1\la+D_0)+(C_1\la+C_0)(
	A_1\la+A_0)^{-1}(B_1\la+B_0)\in\F(\la)^{q\times r}$$ be its transfer function matrix. $\mathcal{L}(\la)$ is a linearization of $G(\la)$ in $\Omega$ if the following conditions hold:
	\begin{itemize}
		\item [(a)] $\mathcal{L}(\la)$ is minimal in $\Omega$, and
		\item[(b)]there exist nonnegative integers $s_1,s_2$ satisfying $s_1-s_2=q-p=r-m,$ and rational matrices $R_1(\la)\in\F(\la)^{(p+s_1)\times (p+s_1)}$ and
		$R_2(\la)\in\F(\la)^{(m+s_1)\times (m+s_1)}$ invertible in $\Omega$ such that
		\begin{equation}\label{equivalencia}
		R_1(\la)\diag(G(\la),I_{s_1})R_2(\la)=\diag(\wh{G}(\la),I_{s_2}).
		\end{equation}
	\end{itemize}
	
\end{deff}

The most usual case is when $q-p=r-m\geq 0.$ This implies that the size of the transfer function matrix $\wh{G}(\la)$ of the linearization $\mathcal{L}(\la)$ is at least the size of the rational matrix $G(\la).$  We assume in the rest of the paper that this holds and take $s_2=0$ and $s:=s_1\geq0$ without loss of generality.

\begin{theo}\cite[Theorem 4.4]{local}[Spectral characterization of linearizations in a subset of $\F$]\label{theo:spectral}
	Let $G(\la)\in\FF(\la)^{p\times m}$ and $\Omega\subseteq\F$ be nonempty. Let
	\[
	\mathcal{L}(\la)=\begin{bmatrix}
	A_1 \la +A_0 &B_1 \la +B_0\\-(C_1 \la +C_0)&D_1 \la +D_0
	\end{bmatrix}\in\efe[\la]^{(n+(p+s))\times (n+(m+s))}
	\]
	be a linear polynomial system matrix, with state matrix $A_1 \la +A_0$, minimal in $\Omega$. Then $\mathcal{L}(\la)$ is a linearization of $G(\la)$ in $\Omega$ if and only if the following three conditions hold:
	\begin{enumerate}
		\item[\rm (a)] $\mbox{\rm nrank} \,   \mathcal{L}(\la)  =\mbox{\rm nrank}\,  G(\la)+n+s $,
		
		\item[\rm (b)]  the pole elementary divisors of $G(\la)$ in $\Omega$ are the elementary divisors of $A_{1}\la+A_{0}$ in $\Omega,$ and
		
		\item[\rm (c)]  the zero elementary divisors of $G(\la)$ in $\Omega$ are the elementary divisors of $\mathcal{L}(\la)$ in $\Omega.$
	\end{enumerate}
\end{theo}

All the concepts and results introduced above are extended to infinity in \cite{local}. For that, the notion of $g$-reversal of a rational matrix \cite[Definition 3.7]{local} is used. That is, given a rational matrix $G(\lambda)\in\F(\lambda)^{p\times m}$ and an integer $g$, the \textit{$g$-reversal} of $G(\la)$ is defined as the rational matrix
\begin{equation*}
\rev_{g} G(\lambda):=\lambda^{g}G\left(\dfrac{1}{\lambda}\right).
\end{equation*}	
A nonconstant linear polynomial system matrix $\mathcal{L}(\la)$ as in (\ref{eq_lin}) is defined to be a \textit{linearization of a rational matrix $G(\la)$ at $\infty$ of grade $g$} if $\rev_1 \mathcal{L}(\la)$ is a linearization of $\rev_{g} G(\la)$ at $0.$  Linearizations at $\infty$  of $G(\la)$ provide information about the pole/zero structure of $G(\la)$ at $\infty$ as explained in \cite[Theorem 4.12 and Proposition 4.13]{local}.

Finally, we recall the notions of rational and minimal bases, which are very useful for the construction of linearizations for rational matrices. A rational matrix $R(\la)$ is said to be a \textit{rational basis}  if its rows form a basis of the rational subspace they span, i.e., if $R(\la)$ has full row normal rank. The term ``rational subspace'' indicates that $\F(\la)$ is the underlying field. Two rational matrices $G(\la) \in \F (\la)^{p\times m}$ and $H(\la) \in \F (\la)^{q\times m}$ are \textit{dual rational bases} if both have full row normal rank, $p+q = m$, and $G(\la) \, H(\la)^T =0$. Let $\Omega\subseteq\F$ be nonempty. We will say that a rational matrix $R(\la) \in \mathbb{F}(\la)^{p \times m}$ has \textit{full row rank in} $\Omega$ if, for all $\la_0\in \Omega$, $R(\la_0)\in \mathbb{F}^{p \times m}$, i.e., $R(\la)$ is defined or bounded at $\la_0$, and $\rank R(\la_0) = p$. Notice that a rational matrix having full row rank in $\F$ must be polynomial, since it has to be defined in $\F.$ For polynomial matrices, the notion of minimal basis is also considered \cite{forney}. One of the most useful characterizations of minimal bases (see \cite[Main Theorem]{forney} or \cite[Theorem 2.2]{BKL}) is that $K(\la)\in\F[\la]^{p\times m}$ is a minimal basis if and only if $K(\la_{0})$ has full row rank for all $\la_{0}\in\F$ and $K(\la)$ is row reduced, i.e., its highest row degree coefficient matrix has full row rank (see \cite[Definition 2.1]{BKL}).

Minimal bases appear in the definition of strong block minimal bases pencils in \cite{BKL}. A strong block minimal bases pencil is a linear polynomial matrix with the following structure
\begin{equation}
\label{eq:minbaspencil}
\begin{array}{cl}
L(\la) =
\left[
\begin{array}{cc}
M(\la) & K_2 (\la)^T \\
K_1 (\lambda) &0
\end{array}
\right]&
\begin{array}{l}
\left. \vphantom{K_2 (\la)^T} \right\} {\scriptstyle p + \widehat{p}}\\
\left. \vphantom{K_1 (\la)} \right\} {\scriptstyle \widehat{m}}
\end{array}\\
\hphantom{\mathcal{L}(\la) =}
\begin{array}{cc}
\underbrace{\hphantom{K_1 (\lambda)}}_{\scriptstyle m + \widehat{m}} & \underbrace{\hphantom{K_2 (\la)^T}}_{\widehat{p}}
\end{array}
\end{array}
\>,
\end{equation}
where $K_1(\la) \in \FF[\la]^{\widehat{m} \times (m + \widehat{m})}$ (respectively $K_2(\la) \in \FF[\la]^{\widehat{p} \times (p + \widehat{p})}$) is a minimal basis with all its row degrees equal to $1$ and with the row degrees of a minimal basis $N_1(\la) \in \FF[\la]^{m \times (m + \widehat{m})}$ (respectively $N_2(\la) \in \FF[\la]^{p \times (p + \widehat{p})}$) dual to $K_1(\la)$ (respectively $K_2(\la)$) all equal. In addition, given a polynomial matrix $D(\la)$, $L(\la)$ is said to be associated to $D(\la)$ if
\begin{equation} \label{eq:Dpolinminbaslin}
D(\la) = N_2(\la) M(\la) N_1(\la)^T.
\end{equation}	
In this case, $L(\la)$ is a strong linearization of $D(\la)$ \cite[Theorem 3.3]{BKL}.

\section{Block full rank linearizations of rational matrices} \label{sec-full-rank-lin}

In this section, we introduce in Theorems \ref{th:1blockfullrank} and \ref{th:1blockfullrankinfinity} a wide family of pencils that give us the information about the zeros and poles of rational matrices locally, at finite points and also at infinity. They will be called block full rank linearizations. Previously, we introduce the notion of block full rank pencils \cite{local}, that are an extension of the strong block minimal bases pencils in \cite{BKL}.

\begin{deff}\cite[Definition 5.2]{local} \label{def:blockfullrank} A block full rank pencil is a linear polynomial matrix over $\mathbb{F}$ with the following structure
	\begin{equation}\label{almost}
	L(\la) =
	\left[
	\begin{array}{cc}
	M(\la) & K_2 (\la)^T \\
	K_1 (\lambda) &0
	\end{array}
	\right]
	\end{equation}
	where $K_1(\la)$ and $K_2(\la)$ are pencils with full row normal rank.
\end{deff}
Definition \ref{def:blockfullrank} includes the cases when $K_1 (\la)$ or $K_2 (\la)$ are empty matrices, that is, $L(\la)=	\left[
\begin{array}{cc}
M(\la) & K_2 (\la)^T
\end{array}
\right]$, $L(\la)=	\left[
\begin{array}{c}
M(\la) \\
K_1 (\lambda)
\end{array}
\right]$ or $L(\la)=M(\la)$.

Theorem 5.3 of \cite{local} states that block full rank pencils can be considered as local linearizations of rational matrices that contain information about their zeros. However, they do not contain information about poles.

\subsection{Block full rank linearizations at finite points}\label{sec:finitepoints}

In Theorem \ref{th:1blockfullrank}, we generalize \cite[Theorem 5.3]{local} in order to obtain local linearizations that give us not only information about the zeros but also about the poles of the corresponding rational matrix. Note also that the linear polynomial system matrix in \eqref{eq:linearization} generalizes the structure of the strong block minimal bases linearizations of rational matrices presented in \cite[Theorem 5.11]{strong} from three perspectives: general pencils $B(\la)$ and $C(\la)$ are allowed, while those in \cite{strong} have a very particular structure; $A(\la)$ can be any regular pencil, while in \cite{strong} its coefficient in $\la$ must be invertible; and $\left[
\begin{array}{cc}
M(\la) & K_2(\la)^T \\
K_1 (\la) & 0
\end{array}
\right]$ is an arbitrary block full rank pencil \eqref{almost}, while in \cite{strong} strong block minimal bases pencils \eqref{eq:minbaspencil} are considered. We give more details in Subsection \ref{SB_asBF}.

\begin{theo} \label{th:1blockfullrank} Consider a nonconstant linear polynomial system matrix \begin{equation} \label{eq:linearization}
	\mathcal{L}(\la) = \left[
	\begin{array}{c|cc}
	A(\la) & \phantom{a} B(\la) \phantom{a} & 0 \\ \hline \phantom{\Big|}
	-C(\la)  \phantom{\Big|}& M(\la) & K_2(\la)^T \\
	0 & K_1 (\la) & 0
	\end{array}
	\right]\in\efe [\la]^{(n+q)\times (n+r)}
	\end{equation}
	with $n>0$ and state matrix $A(\la)$. Let $L(\la):= \left[
	\begin{array}{cc}
	
	M(\la) & K_2(\la)^T \\
	K_1 (\la) & 0
	\end{array}
	\right]$ be a block full rank pencil, and let $N_1 (\la)$ and $N_2 (\la)$ be any rational bases dual to $K_1 (\la)$ and $K_2 (\la)$, respectively. Let $\Omega$ be a nonempty subset of $\mathbb{F}$ such that $K_i (\la)$ and $N_i (\la)$ have full row rank in $\Omega$ for $i=1,2.$ If
	\begin{equation}\label{minimalitycondition}
	\rank\begin{bmatrix} A(\la_0) \\ -N_2(\la_0)C(\la_0)\end{bmatrix}=\rank\begin{bmatrix} A(\la_0) & B(\la_0)N_1(\la_0)^{T} \end{bmatrix}=n
	\end{equation}
	for all $\la_0\in\Omega$ then
	$\mathcal{L}(\la)$ is a linearization of the rational matrix
	\begin{equation}\label{eq:rat_matrix}
	R(\la)=N_{2}(\la) [ M(\la)+C(\la)A(\la)^{-1}B(\la) ] N_{1}(\la)^{T}
	\end{equation} in $\Omega$ with state matrix $A(\la)$.
\end{theo}

A pencil of the form \eqref{eq:linearization} satisfying the hypotheses in Theorem \ref{th:1blockfullrank} is called a \textit{block full rank linearization}. In particular, $\mathcal{L}(\la)$ is said to be a block full rank linearization of $R(\la)$ in $\Omega$ with state matrix $A(\la).$

\begin{rem}\label{rem_emptycasefinite}\rm 	The extreme case of $n=0$ in the linear polynomial system matrix \eqref{eq:linearization} is studied in \cite[Theorem 5.3]{local}. It states that the block full rank pencil $L(\la)$ in Theorem \ref{th:1blockfullrank} is a linearization of the rational matrix $$G(\la) = N_{2}(\la) M(\la) N_{1}(\la)^{T}$$ in $ \Omega$ with empty state matrix. In this case, $L(\la)$ is said to be a block full rank linearization of $G(\la)$ in $\Omega$ with empty state matrix.
\end{rem}

\begin{proof}[Proof of Theorem \ref{th:1blockfullrank}] In order to simplify the notation, throughout this proof we do not specify the sizes of different identity matrices and all of them are denoted by $I$. Let $\widetilde{K}_1 (\la), \widetilde{K}_2 (\la), \widetilde{N}_1 (\la)$ and $\widetilde{N}_2 (\la)$ be minimal bases of the row spaces of $K_1 (\la)$, $K_2 (\la)$, $N_1 (\la)$ and $N_2 (\la)$, respectively. Then, by \cite[Lemma 5.2]{local}, there exist rational matrices $S_1 (\la)$, $S_2 (\la)$, $W_1 (\la)$ and $W_2 (\la)$ such that
	\begin{eqnarray*}
		& K_i (\la) = S_i (\la) \widetilde{K}_i (\la), &  \mbox{and $S_i (\la)$ is invertible in $\Omega$, for $i=1,2,$} \\
		& N_i (\la) = W_i (\la) \widetilde{N}_i (\la), &  \mbox{and $W_i (\la)$ is invertible in $\Omega$, for $i=1,2.$}
	\end{eqnarray*}
	Moreover, $\widetilde{K}_1 (\la), \widetilde{K}_2 (\la), S_1(\la)$ and $S_2(\la)$ are all matrix pencils. We consider the linear polynomial system matrix \begin{equation} \label{eq:linearization2}
	\widetilde{\mathcal{L}}(\la) := \left[
	\begin{array}{c|cc}
	A(\la) & \phantom{a} B(\la) \phantom{a} & 0 \\ \hline \phantom{\Big|}
	-C(\la)  \phantom{\Big|}& M(\la) & \widetilde K_2(\la)^T \\
	0 & \widetilde K_1 (\la) & 0
	\end{array}
	\right],
	\end{equation} which is equivalent in $\Omega$ to $\mathcal{L}(\la),$ since
	$ \left[
	\begin{array}{cc}
	
	I & 0 \\
	0 & S_1(\la)
	\end{array}
	\right]\widetilde{\mathcal{L}}(\la) \left[
	\begin{array}{cc}
	
	I & 0 \\
	0 & S_2(\la)^T
	\end{array}
	\right]=\mathcal{L}(\la).$
	For $i=1,2,$ there exist unimodular matrices
	\[
	U_i(\lambda) =
	\begin{bmatrix}
	\widetilde K_i(\lambda) \\ \widehat{K}_i(\la)
	\end{bmatrix}
	\mbox{, and}\quad
	U_i(\lambda)^{-1}=
	\begin{bmatrix}
	\widehat{N}_i(\lambda)^T & \widetilde N_i(\lambda)^T
	\end{bmatrix}
	\] as in \cite[Theorem 2.10]{BKL}. Consider now the unimodular matrices
	\begin{align*}
	V_{1}(\la) & := \begin{bmatrix}
	I & 0 & 0 & 0 \\
	0 & \widehat{N}_1(\lambda)^T & \widetilde N_1(\lambda)^T & 0 \nonumber \\
	0 & 0 & 0 & I
	\end{bmatrix}  \left[\begin{array}{cccc}
	I & 0 & 0 & 0 \\
	0 & 0 & I & 0 \\
	0 & I & 0 & 0  \\
	0 & -X(\la) & 0 & I
	\end{array}\right], \text{ and} & \\
	V_{2}(\la) & := \left[\begin{array}{cccc}
	I & 0 & 0 & 0 \\
	0 & 0 & I & -Y(\la) \\
	0 & 0 & 0 & I \\
	0 & I & 0 & -Z(\la)
	\end{array}\right]	\begin{bmatrix}
	I & 0 & 0 \\
	0 & \widehat{N}_2(\lambda) & 0 \\ 0 & \widetilde N_2(\lambda) & 0 \\0 & 0 & I
	\end{bmatrix},
	\end{align*} where $Z(\la):=\widehat{N}_2(\lambda)M(\la)\widehat{N}_1(\lambda)^T,$ $X(\la):=\widehat{N}_2(\lambda)M(\la)\widetilde N_1(\lambda)^T,$ $Y(\la):=\widetilde N_2(\lambda)M(\la)\widehat{N}_1(\lambda)^T.$
	We obtain that {\small$$V_{2}(\la)\widetilde{\mathcal{L}}(\la) V_{1}(\la)=  \left[
		\begin{array}{c|ccc}
		A(\la) & \phantom{a} B(\la)\widetilde{N}_1(\lambda)^T \phantom{a} & B(\la)\widehat N_1(\lambda)^T & 0 \\ \hline \phantom{\Big|}
		- \widetilde{N}_2(\lambda)C(\la)  \phantom{\Big|}& \widetilde N_2(\lambda)M(\la)\widetilde{N}_1(\lambda)^T & 0 & 0\\
		0 \phantom{\Big|}& 0 & I & 0 \\
		- \widehat{N}_2(\lambda)C(\la) \phantom{\Big|}&0 & 0 &I
		\end{array}
		\right], $$ }
	which is, in addition, unimodularly equivalent to
	$$ \left[
	\begin{array}{c|ccc}
	A(\la) & \phantom{a} B(\la)\widetilde{N}_1(\lambda)^T \phantom{a} & 0  \\ \hline \phantom{\Big|}
	- \widetilde{N}_2(\lambda)C(\la)  \phantom{\Big|}& \widetilde N_2(\lambda)M(\la)\widetilde{N}_1(\lambda)^T & 0 \\
	0 \phantom{\Big|}& 0 & I  \\
	
	\end{array}
	\right]:=P(\la). $$
	By condition \eqref{minimalitycondition}, the polynomial matrix $$ H(\la):=\left[
	\begin{array}{c|ccc}
	A(\la) & \phantom{a} B(\la)\widetilde{N}_1(\lambda)^T \phantom{a} \\ \hline \phantom{\Big|}
	- \widetilde{N}_2(\lambda)C(\la)  \phantom{\Big|}& \widetilde N_2(\lambda)M(\la)\widetilde{N}_1(\lambda)^T
	
	\end{array}
	\right]$$
	is minimal in $\Omega,$ and its transfer function matrix is
	$$\widetilde R(\la)=\widetilde N_{2}(\la) [ M(\la)+C(\la)A(\la)^{-1}B(\la) ] \widetilde N_{1}(\la)^{T}.$$ Moreover, $ W_2(\la) \widetilde R(\la) W_1(\la)^T = R(\la),$ and, thus, $\widetilde R(\la)$ and $R(\la)$ are equivalent in $\Omega.$ Therefore, the zero elementary divisors of $H(\la)$ in $\Omega$ are the zero elementary divisors of $R(\la)$ in $\Omega,$ and the zero elementary divisors of $A(\la)$ in $\Omega$ are the pole elementary divisors of $R(\la)$ in $\Omega.$ In addition, $P(\la)=\left[
	\begin{array}{cc}
	
	H(\la) & 0 \\
	0 & I
	\end{array}
	\right]$ is unimodularly equivalent to $\widetilde{\mathcal{L}}(\la),$ which is equivalent in $\Omega$ to $\mathcal{L}(\la).$ Therefore, the zero elementary divisors of $\mathcal{L}(\la)$ in $\Omega$ are the zero elementary divisors of $R(\la)$ in $\Omega.$ By Theorem \ref{theo:spectral}, $\mathcal{L}(\la)$ is a linearization in $\Omega$ of $R(\la)$,  since it is immediate to check that the rank condition in Theorem \ref{theo:spectral}$\rm (a)$ is satisfied.
\end{proof}

\begin{rem}\rm Notice that $\mathcal{L}(\la)$ being minimal in $\Omega$ is a necessary condition, but not sufficient, in order property \eqref{minimalitycondition} to be satisfied.
\end{rem}

\begin{rem}\rm \label{rem:empty}
	If in Theorem \ref{th:1blockfullrank}, $K_1(\la)$ (resp. $K_2(\la) $)
	is an empty matrix, we can take the dual rational basis $N_1(\la) $ (resp. $N_2(\la)$) as any rational matrix invertible in $\Omega$ of size the number of colums (resp. rows) of $M(\la).$
\end{rem}

Under the conditions of Theorem \ref{th:1blockfullrank}, Theorem \ref{theo:spectral} guarantees that the elementary divisors of $\mathcal{L}(\la)$ in $\Omega$ are the zero elementary divisors of $R(\la)$ in $\Omega,$ and that the elementary divisors of $A(\la)$ in $\Omega$ are the pole elementary divisors of $R(\la)$ in $\Omega$. 

\begin{example}\rm Let us see a simple example that illustrates Remark \ref{rem:empty}. For instance, for constructing a linearization of  the rational matrix $$R(\la)=\frac{\la-2}{\la+2}\begin{bmatrix}
	\dfrac{-\la +3}{\la^2-1} & \dfrac{1}{\la(\la-1)}
	\end{bmatrix} \text{ in the set } \Omega:=\F-\{-1,0,1\},$$
	we can consider $K_{1}(\la)$ and $K_{2}(\la)$ as empty matrices and dual rational bases $N_{1}(\la)^{T}:=\diag\left(
	\frac{1}{\la^{2}-1} ,	\frac{1}{(\la-1)\la}\right)$ and $N_{2}(\la) := 1$, both invertible in $\Omega$. Then, by Theorem \ref{th:1blockfullrank}, the following linear polynomial system matrix
	
	$$\mathcal{L}(\la):= \left[\begin{array}{c|cc}
	\la+2 & -\la +3 & 1 \\ \hline
	-\la + 2 & 0 & 0
	\end{array}\right]:= \left[\begin{array}{c|c}
	A(\la) & B(\la) \\ \hline
	-C(\la) & M(\la)
	\end{array}\right]$$ is a linearization of $R(\la)$ in $\Omega$ with state matrix $\la+2$, since
	$$\rank\begin{bmatrix}
	\la+2\\
	-\la+2
	\end{bmatrix}=\rank\begin{bmatrix}
	\la + 2 & \dfrac{-\la +3} {\la^2-1} & \dfrac{1}{\la(\la-1)}
	\end{bmatrix}=1$$
	for all $\la\in\Omega$. Therefore, we can recover from $\mathcal{L}(\la)$ the pole and zero structure of $R(\la)$ in $\Omega$. More precisely, $-2$ is the only zero of the state matrix in $\Omega$ and, thus, is the only pole of $R(\la)$ in $\Omega$. Moreover, $2$ is the unique zero of $\mathcal{L}(\la)$ in $\Omega$ and, thus, is the unique zero of $R(\la)$ in $\Omega$.
\end{example}

In Section \ref{automaticlin} we will see that the linearizations for rational approximations of nonlinear eigenvalue problems given in \cite{automatic} are block full rank linearizations.

\begin{rem}\rm We notice that, although the state matrix $A(\la)$ appears in the (1,1) block in \eqref{eq:linearization}, in practice, it can be any regular submatrix of $\mathcal{L}(\la)$. In particular, in  some applications \cite{nlep, automatic} we have found pencils with the structure of block full rank linearizations of the form
	\begin{equation} \label{eq:linearization22}
	\mathcal{L}(\la) = \left[
	\begin{array}{cc|c}
	M(\la) & K_2(\la)^T &  - C(\la) \\
	K_1 (\la) & 0 & 0 \\ \hline \phantom{\Big|}
	B(\la)  & 0  & 	A(\la)
	\end{array}
	\right]\in\efe [\la]^{(q+n)\times (r+n)}.
	\end{equation}
\end{rem}

\subsection{Block full rank linearizations at infinity}\label{sec:infinity}

We now study the counterpart of Theorem \ref{th:1blockfullrank} at infinity. First, we define the notion of degree of a rational matrix. For the scalar case, we define the degree of a rational function $r(\la)=\dfrac{n(\la)}{d(\la)}$ as
\begin{equation}
\deg(r(\la)):= \deg(n(\la))-\deg(d(\la)).
\end{equation}
Then, for rational matrices we consider the following definition (see, for instance, \cite[p.10]{Vard}).
\begin{deff} Let $R(\la)=[r_{ij}(\la)]\in\F(\la)^{p\times m}$ be a rational matrix with entries $r_{ij}(\la)=\dfrac{n_{ij}(\la)}{d_{ij}(\la)}$. The degree of $R(\la)$ is then defined as
	$$\deg(R(\la)):= \max_{\substack{i=1,\ldots,p\\j=1,\ldots,m}} \{\deg(r_{ij}(\la))\}.$$
\end{deff}
Notice that this notion of degree of a rational matrix generalizes the notion of degree of a polynomial matrix. In what follows, we call the degrees of each row of $R(\la)$, the row degrees of $R(\la)$.

\begin{lem}\label{lemma_t} Let $R(\la)$ be a rational matrix. Then there exists an integer $t$ such that all the rows of $\rev_t R(\la)$ are defined at $0$ and are  all different from zero at $0$ if and only if all the row degrees of $R(\la)$ are equal to $t$.
\end{lem}

\begin{proof} First, we consider a rational function $r(\la)=\dfrac{a(\la)}{b(\la)}$
	such that the numerator $a(\la)$ has degree $n$, and that the denominator $b(\la)$ has degree $m$. We assume that there exists an integer $t$ for which $\rev_t r(0)$ is defined and is different from $0$. We can write
	\begin{equation}\label{t_reversal}
	\rev_t r(\la)=\la^{t+m-n}\dfrac{\rev_n a(\la)}{\rev_m b(\la)}:=\la^{t+m-n}h(\la),
	\end{equation}	
	where $0$ is not a pole nor a zero of $h(\la)$ since $\rev_n a(0)\neq 0$ and $\rev_m b(0)\neq 0$. That is, $h(\la)$ is defined and is different from $0$ at $0$. Therefore, $t+m-n=0$, so that $\rev_t r(\la)$ is also defined and is different from $0$ at $0$. Then, $t=n-m=\deg(r(\la))$. Now, we assume that we have a rational row vector $v(\la)=[r_{1}(\la)\quad \cdots \quad r_{s}(\la)]$ such that, for some integer $t$, $\rev_t v(0)$ is defined and is different from $0$. Then, it must be $t=\displaystyle\max_{i=1,\ldots,s} \{\deg(r_{i}(\la))\}$. That is, $t=\deg(v(\la))$. Finally, consider a rational matrix $R(\la)=[v_{1}(\la)^{T}\quad \cdots \quad v_{q}(\la)^{T}]^{T}$ where $v_{i}(\la)$ are rational row vectors. Assume that there exists an integer $t$ such that all the rows of $\rev_t R(\la)$ are defined at $0$ and are different from zero at $0$. Then, for each row $v_{i}(\la)$, it must hold that $t=\deg(v_{i}(\la))$.
	
	The converse is trivial taking into account equation \eqref{t_reversal} for each entry of $R(\la)$.
\end{proof}

\begin{theo} \label{th:1blockfullrankinfinity} Consider a nonconstant linear polynomial system matrix \begin{equation} \label{eq:linearizationinf}
	\mathcal{L}(\la) = \left[
	\begin{array}{c|cc}
	A(\la) & \phantom{a} B(\la) \phantom{a} & 0 \\ \hline \phantom{\Big|}
	-C(\la)  \phantom{\Big|}& M(\la) & K_2(\la)^T \\
	0 & K_1 (\la) & 0
	\end{array}
	\right]\in\efe [\la]^{(n+q)\times (n+r)}
	\end{equation} with $n>0$ and state matrix $A(\la)$. Let $L(\la):= \left[
	\begin{array}{cc}
	
	M(\la) & K_2(\la)^T \\
	K_1 (\la) & 0
	\end{array}
	\right]$ be a block full rank pencil and let $N_1 (\la)$ (resp. $N_2 (\la)$) be any rational basis dual to $K_1 (\la)$ (resp. $K_2 (\la)$) with its row degrees all equal to an integer $t_1$ (resp. $t_2$). If $\rev_1 K_i (\la)$ and $\rev_{t_i} N_i (\la)$ have full row rank at zero for $i=1,2$ and
	\begin{equation}\label{minimalitycondition_inf}
	\rank\begin{bmatrix} \rev_1 A(0) \\ -\rev_{t_2} N_2(0)\rev_1 C(0)\end{bmatrix}=\rank\begin{bmatrix} \rev_1 A(0) & \rev_1 B(0) \rev_{t_1}N_1(0)^{T} \end{bmatrix}=n
	\end{equation}
	then
	$\mathcal{L}(\la)$ is a linearization  of the rational matrix
	\begin{equation*}
	R(\la)=N_{2}(\la) [ M(\la)+C(\la)A(\la)^{-1}B(\la) ] N_{1}(\la)^{T}
	\end{equation*}
	at $\infty$ of grade $1 + t_1 + t_2$ with state matrix $A(\la)$.
	
\end{theo}

A pencil of the form \eqref{eq:linearizationinf} satisfying the hypotheses in Theorem \ref{th:1blockfullrankinfinity} is called \textit{block full rank linearization at infinity}. In particular, $\mathcal{L}(\la)$ is said to be a block full rank linearization of $R(\la)$ at $\infty$ of grade $1+t_1+t_2$ with state matrix $A(\la)$. In general, a block full rank linearization is said to be \textit{strong} if it is a linearization in $\F\cup \{\infty\}.$

\begin{rem}\label{rem_emptycaseinfinity}\rm The extreme case of $n=0$ in the linear polynomial system matrix \eqref{eq:linearizationinf} is studied in \cite[Theorem 5.5]{local}. It states that the block full rank pencil $L(\la)$ in Theorem \ref{th:1blockfullrankinfinity} is a linearization  of the rational matrix $$G(\la) = N_{2}(\la) M(\la) N_{1}(\la)^{T}$$ at $\infty$ of grade $1 + t_1 + t_2$ with empty state matrix. In this case, $L(\la)$ is said to be a block full rank linearization of $G(\la)$ at $\infty$ of grade $1+t_1+t_2$ with empty state matrix. We notice that, by Lemma \ref{lemma_t}, the integers $t_{1}$ and $t_{2}$ appearing in \cite[Theorem 5.5]{local} are the row degrees of the dual bases $N_{1}(\la)$ and $N_{2}(\la)$, respectively.
\end{rem}

\begin{proof}[Proof of Theorem \ref{th:1blockfullrankinfinity}] The result follows by applying Theorem \ref{th:1blockfullrank} to $\rev_1 	\mathcal{L}(\la) $ at $0.$ Then $\rev_1 	\mathcal{L}(\la) $ is a linearization at $0$ of $$\rev_{t_2}N_{2}(\la) [ \rev_1 M(\la)+\rev_1 C(\la)(\rev_1 A(\la))^{-1}\rev_1 B(\la) ] \rev_{t_1} N_{1}(\la)^{T}= \rev_{1+t_1+t_2}R(\la).$$
\end{proof}

\begin{example} \rm We now consider the rational matrix $$R(\la):= \displaystyle\sum_{k=0}^{2} A_{k} \frac{\la^{k}}{(\la-\epsilon)^{2}} +I_n \frac{1}{\la}\in\F(\la)^{n\times n} \text{ and the set }\Omega:=\F-\{\epsilon\},$$  for some $\epsilon\in\F$. Then we define the following linear polynomial system matrix
	$$\mathcal{L}(\la):=\left[\begin{array}{c|cc|c}
	\la I_n & 0 & (\la-\epsilon)I_n & 0 \\ \hline
	0 & A_{2} & 0 & -I_n \\
	-(\la-\epsilon)I_n & 0 & \la A_1 + A_0 & \la I_n \\ \hline
	0 & -I_n & \la I_n & 0
	\end{array}\right]=:\left[\begin{array}{c|c|c}
	A(\la) & B(\la) & 0 \\ \hline
	- C(\la ) & M(\la) & K_{2}(\la)^{T} \\ \hline
	0 & K_{1}(\la) & 0
	\end{array}\right],$$
	with state matrix $A(\la)=\la I_n$. We consider the dual rational bases $N_{1}(\la):=N_{2}(\la):=\begin{bmatrix}
	\dfrac{\la I_n}{\la-\epsilon} & \dfrac{I_n}{\la - \epsilon}
	\end{bmatrix}$, which have row degrees $t_{1}=t_2=0$. Then $R(\la)$ can be written as $R(\la)=N_{2}(\la) [ M(\la)+C(\la)A(\la)^{-1}B(\la) ] N_{1}(\la)^{T}$. Notice that $\rev_{t_{i}}N_{i}(\la)$ and $\rev_1 K_i(\la)$ have both full row rank at $0$, for $i=1,2$, and that condition \eqref{minimalitycondition_inf} is satisfied since $\rev_1 A(0)=I_n$. Thus, by Theorem \ref{th:1blockfullrankinfinity}, $\mathcal{L}(\la)$ is a linearization of $R(\la)$ at $\infty$ of grade $1+t_1+t_2=1 $ with state matrix $A(\la)$. In addition, by Theorem \ref{th:1blockfullrank}, $\mathcal{L}(\la)$ is a linearization of $R(\la)$ in $\Omega$, since $N_{i}(\la)$ and $K_i(\la)$ have both full row rank in $\Omega$, for $i=1,2$, and condition \eqref{minimalitycondition} is satisfied in $\Omega$. Observe that, if $R(\la)$ has symmetric coefficients, $\mathcal{L}(\la)$ preserves the symmetry.
	
\end{example}

\begin{rem}\rm If we want a linearization as in \eqref{eq:linearizationinf} to be a linearization at all finite and infinite points we need, besides minimality conditions, the matrices $K_{1}(\la)$ and $K_2(\la)$ being minimal bases with all their row degrees equal to 1. Notice that if a pencil $K(\la)$ has full row rank in $\F$ and, in addition, $\rev_1 K(\la)$ has full row rank at $0$ then $K(\la)$ is a minimal basis. Conversely, if $K(\la)$ is a minimal basis with all its row degrees equal to one then $K(\la)$ has full row rank in $\F$ and $\rev_1 K(\la)$ has full row rank at $0$.
\end{rem}

\subsection{Strong block minimal bases linearizations as block full rank linearizations}\label{SB_asBF}

By using strong block minimal bases pencils, in \cite[Theorem 5.11]{strong}  (strong) linearizations are constructed that contain the complete spectral information of rational matrices, finite and infinite,  as well as the information about their minimal bases and indices \cite{minimal}, when the corresponding rational matrix $R(\la)$ is expressed in the form $R(\la) = D(\la) + C (\la I_n - A)^{-1} B$, where  $D(\la)$ is the polynomial part of $R(\la)$ with $\deg (D(\la)) >1,$ and $C (\la I_n - A)^{-1} B$ is a minimal state-space realization of the strictly proper part of $R(\la).$ We will see that such linearizations satisfy Theorems \ref{th:1blockfullrank}, with $\Omega=\F$, and \ref{th:1blockfullrankinfinity}.

For the construction, in \cite{strong} $
L(\la) :=
\left[
\begin{array}{cc}
M(\la) & K_2 (\la)^T \\
K_1 (\lambda) &0
\end{array}
\right]
$
is considered to be a strong block minimal bases pencil  as in \eqref{eq:minbaspencil} associated to $D(\la)$ with sharp degree, that is, such that $\deg(D(\la)) = \deg(N_2(\la)) +  \deg(N_1(\la)) + 1$, where $N_1(\la) $ (respectively $N_2(\la)$) is a minimal basis dual to $K_1(\la)$ (respectively $K_2(\la)$). Then, constant matrices $\widehat{K}_1 \in \FF^{m \times (m + \widehat{m})}$ and $\widehat{K}_2 \in \FF^{p \times (p + \widehat{p})}$ and matrices $\widehat{N}_1(\lambda)^{T} \in \FF[\lambda]^{(m + \widehat{m})\times \widehat{m}}$ and $\widehat{N}_2(\lambda)^{T} \in\FF[\lambda]^{(p + \widehat{p})\times\widehat{p}}$ exist such that, for $i=1,2$,
\[
U_i(\lambda) =
\begin{bmatrix}
K_i(\lambda) \\ \widehat{K}_i
\end{bmatrix}
\quad \mbox{and} \quad
U_i(\lambda)^{-1}=
\begin{bmatrix}
\widehat{N}_i(\lambda)^T & N_i(\lambda)^T
\end{bmatrix}
\]
are unimodular. Finally, the following linear polynomial matrix is constructed
\begin{equation} \label{eq:ratstrongblockmin}
\mathcal{L}(\la) = \left[
\begin{array}{c|cc}
X(\la I_n -A)Y & \phantom{a} XB\widehat{K}_1 \phantom{a} & 0 \\ \hline \phantom{\Big|}
-\widehat{K}_2^T C Y \phantom{\Big|}& M(\la) & K_2(\la)^T \\
0 & K_1 (\la) & 0
\end{array}
\right],
\end{equation}
where $X, Y \in\FF^{n\times n}$ are any nonsingular constant matrices. With  these assumptions, $\mathcal{L}(\la)$
is a strong linearization of $R(\la)$ \cite[Theorem 5.11]{strong}.

This result follows as a corollary of Theorems \ref{th:1blockfullrank} and \ref{th:1blockfullrankinfinity} as well. First, since $\widehat{K}_iN_i(\la)^{T}=I$,  notice that $R(\la)$ can be written as in \eqref{eq:rat_matrix}:
$$	N_{2}(\la) [ M(\la)	+\widehat{K}_2^T C Y Y^{-1}(\la I_n-A)^{-1}X^{-1}XB\widehat K_{1} ] N_{1}(\la)^{T} = D(\la) + C (\la I_n - A)^{-1} B =R(\la),$$
and, in addition, conditions \eqref{minimalitycondition}, with $\Omega =\F$, and \eqref{minimalitycondition_inf} are satisfied. More precisely, we have that
\begin{equation*}
\rank\begin{bmatrix} X(\la I_n -A)Y \\ -N_2(\la)\widehat{K}_2^T C Y\end{bmatrix}=\rank\begin{bmatrix} X( \la I_n -A)Y \\ -C Y\end{bmatrix}= n, \text{and}
\end{equation*}
\begin{equation*}
\rank\begin{bmatrix} X(\la I_n -A)Y & XB\widehat{K}_1N_1(\la)^{T} \end{bmatrix}=\rank\begin{bmatrix} X(\la I_n -A)Y & XB \end{bmatrix}=n
\end{equation*}
for all $\la\in\F,$ since 
$X$ and $Y$ are nonsingular, and the realization $C (\la I_n - A)^{-1} B$ is minimal. Therefore, condition \eqref{minimalitycondition} is satisfied and, thus, $\mathcal{L}(\la)$
is a linearization of $R(\la)$ in $\F$ by Theorem \ref{th:1blockfullrank}. Moreover, we have that condition \eqref{minimalitycondition_inf} is satisfied since $\rev_1 (\la I_n - A)$ evaluated at $0$ is just $I_n$. Then, by Theorem \ref{th:1blockfullrankinfinity},  $\mathcal{L}(\la)$ is a linearization of $R(\la)$ at infinity of grade $\deg(N_2(\la)) +  \deg(N_1(\la)) + 1=\deg(D(\la)) $.

\section{Block full rank linearizations for ``the automatic rational approximations of NLEPs'' in \cite{automatic}}\label{automaticlin}

In this section we study the  precise properties of the linearizations for rational approximations of nonlinear eigenvalue problems (NLEPs) in \cite{automatic}. We will see that they satisfy Theorem \ref{th:1blockfullrank} in a particular subset of $\C$, where $\C$ is the field of complex numbers, and in the whole field $\C$ under mild conditions.

We start by defining a NLEP \cite{guttel-tisseur-2017}. Given a non-empty open set $\Omega \subseteq \C$ and an analytic matrix-valued function
\[
\begin{array}{cccc}
F : & \Omega & \rightarrow &  \mathbb{C}^{n\times n} \\
& \la & \mapsto &  F(\la),
\end{array}
\]
the NLEP consists of computing scalars $\la_0 \in \Omega$ (eigenvalues) and nonzero vectors $v \in \C^n$ (eigenvectors) such that
\[ \begin{array}{c}  \phantom{\Big(}
F(\la_0) v = 0, \phantom{\Big)}\end{array}
\]
under the regularity assumption $\det (F(\la)) \not\equiv 0$. Since a direct solution of NLEPs is usually not possible, they are often solved  numerically via rational approximation and by solving the corresponding REP with linearizations adapted to the structure of the obtained rational matrix. An approach to obtain  an automatic rational approximation for the NLEP  in a region $\Omega$ is given in \cite{automatic}. The authors of \cite{automatic} consider the $n\times n$ nonlinear matrix $F(\la)$ of the NLEP written in the form
\begin{equation}\label{nonlinearmatrix}
F(\la)=Q(\la)+\displaystyle\sum_{i=1}^{s}(C_{i}-\la D_{i})g_{i}(\la)
\end{equation}
with $Q(\la)$ a polynomial matrix,  $C_i$ and $D_i$ $n\times n$ constant matrices, and $g_{i}(\la)$ nonlinear scalar functions. Then, a CORK linearization of $Q(\la)$  of those introduced in \cite{cork} is considered, and the functions $g_{i}(\la)$ are approximated by rational functions employing the adaptive Antoulas--Anderson (AAA) algorithm \cite{AAA}, or its set-valued generalization  presented in \cite{automatic}. We recall the definition of CORK linearization as given in \cite{automatic}.

\begin{deff}
	Let $G(\la)$ be an $n\times n$ rational matrix written as
	\begin{equation}\label{rationalmatrix}
	G(\la)=\displaystyle\sum_{i=0}^{k-1}(A_{i}-\la B_{i})f_{i}(\la),
	\end{equation}
	where $f_{i}(\la)$ are scalar rational functions with $f_{0}(\la) \equiv 1$,  and $A_i$, $B_i$ are $n\times n$ constant matrices. Define $$f(\la):=[f_{0}(\la)\quad \cdots \quad f_{k-1}(\la)]^{T},$$ and assume that the rational functions $f_i(\la)$ satisfy a linear relation
	\begin{equation}\label{linearrelation}
	(X-\la Y)f(\la)=0,
	\end{equation}
	where $\rank(X-\la Y)=k-1\text{  for all }\la\in\C,$ and $X-\la Y$ has size $(k-1)\times k.$ Then the matrix pencil
	$$ \mathcal{L}_{G}(\la)=\left[\begin{array}{c}\begin{array}{ccc}
	A_{0}-\la B_{0} & \cdots & A_{k-1}-\la B_{k-1}
	\end{array} \\ \hdashline[2pt/2pt] \phantom{\Big|}
	(X-\la Y)\otimes I_{n}
	\end{array}\right]  $$
	is called a CORK linearization of $G(\la)$.
\end{deff}

If the vector $f(\la)$ is polynomial then $G(\la)$ in \eqref{rationalmatrix} is a polynomial matrix and $\mathcal{L}_{G}(\la)$ is a linearization of $G(\la)$ in $\mathbb{C}$, in particular, $\mathcal{L}_{G}(\la)$ is a block full rank linearization of $G(\la)$ in $\C$ with empty state matrix. However, if $f(\la)$ is a rational vector then $ \mathcal{L}_{G}(\la)$ is not, in general, a linearization in the sense of \cite{strong}, that is, in $\C$ but it is a linearization in a local sense \cite{local}. More precisely, it is a block full rank linearization in all the subsets where the rational vector $f(\la)$ is defined, i.e., has no poles. Such result is stated in the next theorem.

\begin{theo}\label{th_CORK} Let $\Omega$ be a nonempty subset of $\C$ where the rational vector $f(\la)$ in \eqref{linearrelation} is defined. Then a CORK linearization $ \mathcal{L}_{G}(\la)$ of a rational matrix $G(\la)$ as in \eqref{rationalmatrix} is a block full rank linearizacion of $G(\la)$ in $\Omega$ with empty state matrix.
	
\end{theo}

\begin{proof} Notice that, by taking $M(\la):=\left[\begin{array}{ccc}
	A_{0}-\la B_{0} & \cdots & A_{k-1}-\la B_{k-1}\end{array}\right],$ $K_1(\la):=
	(X-\la Y)\otimes I_{n},$ and $K_2(\la)$ empty, $\mathcal{L}_{G}(\la)$ is a block full rank pencil. Moreover, $f(\la)^T$ is a rational basis dual to $X-\la Y.$ Then we apply Remark \ref{rem_emptycasefinite} with $N_1(\la):=f(\la)^T\otimes I_n$ and  $N_2(\la)=I_n$.
\end{proof}

Once CORK linearizations and some of their properties have been revised, we recall the AAA approximation of scalar functions. A given nonlinear function $g: \C \longrightarrow \C $ is approximated in \cite{automatic} on a set $\Sigma\subset \C$ by a rational function $r(\la)$ in barycentric form, that is,
\begin{equation}\label{barycentricaprox}
r(\la)=\displaystyle\sum_{j=1}^{m}\frac{g(z_{j})w_{j}}{\la-z_{j}}\Big/\displaystyle\sum_{j=1}^{m}\frac{w_{j}}{\la-z_{j}},
\end{equation} where $z_{1},\ldots,z_{m}$ are distinct support points and $w_{1},\ldots,w_{m}$ are nonzero weights,  that can be automatically chosen as explained in \cite{AAA}. In this case, $\displaystyle\lim_{\la \to z_{j}}r(\la)=g(z_{j}).$ It is shown in \cite[Proposition 2.1]{automatic} that \eqref{barycentricaprox} can be written as
$$r(\la)=\left[g(z_{1})w_{1}\quad \cdots \quad g(z_{m})w_{m}\right]\left[\begin{array}{ccccc}
w_{1} & w_{2} & \cdots & w_{m-1} & w_{m} \\
\la-z_{1} & z_{2}-\la \\
& \la-z_{2} & \ddots \\
& & \ddots & z_{m-1}-\la \\
& & & \la-z_{m-1} & z_{m}-\la
\end{array}\right]^{-1} \left[\begin{array}{c}
1\\
0\\
\vdots\\
0
\end{array}\right].$$
That is, $r(\la)$ can be written as a generalized state-space realization. Then, the pencil
\begin{equation} \label{polsysmat_escalarrational}
{\small P(\la) := \left[
	\begin{array}{ccccc|c}
	w_{1} & w_{2} & \cdots & w_{m-1} & w_{m} & -1 \\
	\la-z_{1} & z_{2}-\la & & & & 0 \\
	& \la-z_{2} & \ddots & & &  \vdots \\
	& & \ddots & z_{m-1}-\la & & \vdots \\
	& & & \la-z_{m-1} & z_{m}-\la & 0 \\ \hline \phantom{\Big|}
	g(z_{1})w_{1} & g(z_{2})w_{2} & \cdots & g(z_{m-1})w_{m-1} & g(z_{m})w_{m} & 0
	\phantom{\Big|} \end{array}
	\right]:=\left[
	\begin{array}{c|c}
	E-\la F & -b\\ \hline  \phantom{\Big|}
	a^{T} & 0
	\end{array}
	\right]}
\end{equation} is a linear polynomial system matrix of $r(\la)$ (i.e., with transfer function $r(\la)$) by considering $E-\la F$ as state matrix. In order to know what pole and zero information of $r(\la)$ we can obtain from this realization, we consider in Proposition \ref{minimality_AAA} the polynomial system matrix $P(\la)$ and we study its minimality.  First, we prove Lemma \ref{lemma_statematrix}. In both Proposition \ref{minimality_AAA} and Lemma \ref{lemma_statematrix}, we consider $r(\la)$ written as the following quotient of polynomials
\begin{equation}\label{quotient}
r(\la):=\dfrac{p(\la)}{q(\la)},
\end{equation}
where $$p(\la):=\left(\displaystyle\prod_{j=1}^{m}(\la-z_{j})\right)\left(\displaystyle\sum_{j=1}^{m}\frac{g(z_{j})w_{j}}{\la-z_{j}}\right) \text{ and } q(\la):=\left(\displaystyle\prod_{j=1}^{m}(\la-z_{j})\right)\left(\displaystyle\sum_{j=1}^{m}\frac{w_{j}}{\la-z_{j}}\right).$$ Note that the representation of the rational function \eqref{quotient}  might not be irreducible. We will see that the irreducibility of \eqref{quotient} is the key property for the minimality of $P(\la)$.

\begin{lem}\label{lemma_statematrix}  The pencil $E-\la F$ in \eqref{polsysmat_escalarrational} is a strong block minimal bases pencil associated to  the polynomial $q(\la)$ in \eqref{quotient}. \end{lem}

\begin{proof} We set
	\begin{equation}\label{eq:E-laF} {\small
		E-\la F=\left[
		\begin{array}{ccccc}
		w_{1} & w_{2} & \cdots & w_{m-1} & w_{m}  \\ \hline
		\la-z_{1} & z_{2}-\la & & &  \\
		& \la-z_{2} & \ddots & &  \\
		& & \ddots & z_{m-1}-\la &  \\
		& & & \la-z_{m-1} & z_{m}-\la  \end{array}
		\right]=:\left[
		\begin{array}{c}
		M \\ \hline \phantom{\Big|}
		K(\la)
		\end{array}
		\right].}
	\end{equation}
	Since $z_{1},\ldots,z_{m}$ are distinct, $K(\lambda_{0})$ has full row rank for all $\lambda_{0}\in\C$. In addition, $K(\lambda)$ is row reduced because its highest row degree coefficient matrix $$K_{hr}=\left[ {\begin{array}{cccccc}
		1 & -1 & 0 &  \\
		& 1& -1 & 0 &  \\
		& &\ddots &\ddots & \ddots &  \\
		& & &1& -1 & 0 \\
		& & & & 1 & -1
		\end{array} } \right] $$
	has full row rank. We conclude that $K(\lambda)$ is a minimal basis. Let us denote
	\begin{equation}\label{eq_dual}
	{\small N(\la):=\displaystyle\prod_{j=1}^{m}(\la-z_{j}) \left[\frac{1}{\la-z_1}\quad \cdots \quad \frac{1}{\la-z_m} \right].}
	\end{equation}
	Then, it is not difficult to prove that $N(\lambda)$ is also a minimal basis, taking again into account that $z_{1},\ldots,z_{m}$ are distinct. Moreover, since $K(\lambda)N(\lambda)^{T}=0$ and $\left[ {\begin{array}{cc}
		K(\lambda) \\
		N(\lambda) \\
		\end{array} } \right]$ is a square matrix, we have that $K(\lambda)$ and $N(\lambda)$ are dual minimal bases. In addition, all the row degrees of $K(\lambda)$ are equal to $1$ and the row degree of $N(\lambda)$ is equal to $m-1.$ Hence, $E-\la F$ is a strong block minimal bases pencil associated to the polynomial matrix $MN(\lambda)^{T}=q(\la)$.	
\end{proof}

$E-\la F$ being a strong block minimal bases pencil associated to the polynomial $q(\la)$ implies that $E-\la F$ is a  (strong) linearization of $q(\la)$ and, in particular, that the determinant of $E-\la F$ is equal to $q(\la)$ up to a scalar multiple. This fact is used  to prove the following result.

\begin{prop}\label{minimality_AAA} Consider the rational function $r(\la)$ in \eqref{quotient} and its linear polynomial system matrix $P(\la)$ in \eqref{polsysmat_escalarrational}. Then, $P(\la)$ is not minimal at $\la_0\in\C$ if and only if $\la_0$ is a zero of both polynomials $p(\la)$ and $q(\la)$.
\end{prop}

\begin{proof} Consider $P(\la) =\left[
	\begin{array}{c|c}
	E-\la F & -b\\ \hline  \phantom{\Big|}
	a^{T} & 0
	\end{array}
	\right]$ as in \eqref{polsysmat_escalarrational}. By Lemma \ref{lemma_statematrix},
	\begin{equation}\label{det_statemat}
	\det (E-\la F)=\alpha q(\la)\text{ with }\alpha\neq 0.
	\end{equation}
	In addition, since the Schur complement of $E-\la F$ in $P(\la)$ is $r(\la)$, we have that
	\begin{equation}\label{det_polysystemmat}
	\det(P(\la))=\det(E-\la F) r(\la)=\alpha q(\la)\dfrac{p(\la)}{q(\la)}=\alpha p(\la).
	\end{equation}
	Now, assume that $\la_0$ is a zero of both polynomials $p(\la)$ and $q(\la)$. That is, we can cancel out at least one factor of the form $(\la-\la_0)$ in both numerator and denominator of $r(\la)$. Then, the algebraic multiplicity of $\la_0$ as a zero of $r(\la)$ is not the same as the algebraic multiplicity of $\la_0$ as a zero of $P(\la)$. Therefore, $P(\la)$ is not minimal at $\la_0$. For the converse, assume that $P(\la)$ is not minimal at $\la_0$. Then, $\rank \begin{bmatrix}
	E-\la_0 F \\
	a^T
	\end{bmatrix}<m$, since $\rank \begin{bmatrix}
	E-\la_0 F & -b
	\end{bmatrix}=m$. Actually, $\rank \begin{bmatrix}
	E-\la_0 F \\
	a^T
	\end{bmatrix}=m-1$, as the sub-matrix $${\small	K(\la_0)=\left[\begin{array}{ccccc}
		\la_0-z_{1} & z_{2}-\la_0 & & &  \\
		& \la_0-z_{2} & \ddots & &  \\
		& & \ddots & z_{m-1}-\la_0 &  \\
		& & & \la_0-z_{m-1} & z_{m}-\la_0  \end{array}
		\right]}$$ contains a non-zero minor of order $m-1$ for all $\la_0\in\C$, where $K(\la)$ is the matrix appearing in \eqref{eq:E-laF}. Therefore, by using the notation $\mathcal{N}_r(\cdot)$ for the right nullspace, $\dim \mathcal{N}_r \left(\begin{bmatrix}
	E-\la_0 F \\
	a^T
	\end{bmatrix}\right) =1$ and $\dim \mathcal{N}_r (K(\la_0)) =1$. Then, $\mathcal{N}_r \left(\begin{bmatrix}
	E-\la_0 F \\
	a^T
	\end{bmatrix}\right) = \mathcal{N}_r (K(\la_0))$, since  $\mathcal{N}_r \left( \begin{bmatrix}
	E-\la_0 F \\
	a^T
	\end{bmatrix}\right) \subseteq \mathcal{N}_r (K(\la_0))$ and both have the same dimension.  Actually,  $\mathcal{N}_r (K(\la_0))=\text{Span}\{N(\la_0)^{T}\}$, where $N(\la)$ is the polynomial matrix in \eqref{eq_dual}. Hence, $\begin{bmatrix}
	E-\la_0 F \\
	a^T
	\end{bmatrix} N(\la_0)^{T} =0$ and, therefore, $\left[
	\begin{array}{ccccc}
	w_{1} & w_{2} & \cdots & w_{m} \end{array}\right]N(\la_0)^{T}=0$ and $\left[
	\begin{array}{ccccc}
	g(z_1)w_{1} & g(z_2)w_{2} & \cdots & g(z_m)w_{m} \end{array}\right]N(\la_0)^{T}=0$. That is, $\la_0$ is a root of both $q(\la)$ and $p(\la)$.
	
\end{proof}

With these tools at hand, we go back to the original NLEP. Let $F(\la)$ be the nonlinear matrix function in \eqref{nonlinearmatrix}. Then, each function $g_{i}(\la)$ is approximated in \cite{automatic} on a set $\Sigma\subset \C$ by a rational function $r_{i}(\la)$ as in \eqref{barycentricaprox}, i.e.,
\begin{equation*}
g_{i}(\la)\approx r_{i}(\la)=\displaystyle\sum_{j=1}^{\ell_{i}}\frac{g_{i}(z_{j}^{i})w_{j}^{i}}{\la-z_{j}^{i}}\Big/\displaystyle\sum_{j=1}^{\ell_i}\frac{w_{j}^{i}}{\la-z_{j}^{i}},
\end{equation*}
where $\ell_{i}$ is the number of support points $z_{j}^{i}$ and weights $w_{j}^{i}$ for each $i=1,\ldots,s.$ For that, one can use the AAA algorithm on each function $g_{i}(\la)$ separately \cite{AAA}, or one can use the set-valued AAA algorithm in \cite[Section 2.2]{automatic}, so that the rational approximants $r_{i}(\la)$ are all constructed simultaneously and sharing the same support points $z_{j}^{i}:=z_{j}$ and weights $w_{j}^{i}:= w_{j}$. By using any of the two approaches above, the following approximation of $F(\la)$ on $\Sigma$ is obtained:
\begin{equation}\label{aproxmatrix}
F(\la)\approx R(\la):= Q(\la)+\displaystyle\sum_{i=1}^{s}(C_{i}-\la D_{i})r_{i}(\la).
\end{equation}
Next, the polynomial matrix $Q(\la)$ is expressed in the form of \eqref{rationalmatrix},  i.e., $Q(\la):= \displaystyle\sum_{i=0}^{k-1}(A_{i}-\la B_{i})f_{i}(\la)$, assuming the functions $f_i(\la)$ are polynomials, with $f_0(\la)=1$, and each $r_{i}(\la)$ is written in generalized state-space form, that is,
\begin{equation}\label{rational}
R(\la)= \displaystyle\sum_{i=0}^{k-1}(A_{i}-\la B_{i})f_{i}(\la) +\displaystyle\sum_{i=1}^{s}(C_{i}-\la D_{i})a_{i}^{T}(E_{i}-\la F_{i})^{-1}b_{i},
\end{equation}
with $a_{i}=\left[g_i(z_{1}^{i})w_{1}^{i}\quad \cdots \quad g_i(z_{\ell_{i}}^{i})w_{\ell_{i}}^{i}\right]^{T}\in\C^{\ell_{i}}$, $b_{i}=\left[1\quad 0 \quad \cdots \quad 0\right]^{T} \in\C^{\ell_{i}}$ and $\ell_{i}\times \ell_{i}$ matrices
$$E_{i}=\left[\begin{array}{ccccc}
w_{1}^{i} & w_{2}^{i} & \cdots & w_{\ell_{i}-1}^{i} & w_{\ell_{i}}^{i} \\
-z_{1}^{i} & z_{2}^{i} \\
& -z_{2}^{i} & \ddots \\
& & \ddots & z_{\ell_{i}-1}^{i} \\
& & & -z_{\ell_{i}-1}^{i} & z_{\ell_{i}}^{i}
\end{array}\right]\text{  and  }F_{i}=\left[\begin{array}{ccccc}
0 & 0 & \cdots & 0 & 0 \\
-1 & 1 \\
& -1 & \ddots \\
& & \ddots & 1 \\
& & & -1 & 1
\end{array}\right]. $$
The linearization constructed in \cite{automatic} for $R(\la)$ is the following.

\begin{deff}\cite[Definition 3.2]{automatic}(CORK linearization for AAA rational approximation)\label{corklinearization}
	Let $R(\la)$ be a rational matrix as in \eqref{rational}. Consider $b:=[b_{1}^{T}\quad \cdots \quad b_{s}^{T}]^{T}$ and $E-\la F:=\diag(E_{1}-\la F_{1},\ldots,E_{s}-\la F_{s})$. Then a CORK linearization for $R(\la)$ is
	$$\begin{small} \mathcal{L}_{R}(\la)=\left[\begin{array}{c|c}\begin{array}{ccc}
	A_{0}-\la B_{0} & \cdots & A_{k-1}-\la B_{k-1}
	\end{array} & \begin{array}{ccc}
	a_{1}^{T}\otimes (C_{1}-\la D_{1}) & \cdots & a_{s}^{T}\otimes (C_{s}-\la D_{s})
	\end{array} \\ \hdashline[2pt/2pt] \phantom{\Big|}
	(X-\la Y)\otimes I_{n} & 0 \\ \hline \phantom{\Big|}
	\begin{array}{cccccccccccc}  -b\otimes I_{n} & & &  & & 0 & & & & \end{array} & (E-\la F)\otimes I_{n}
	\end{array}\right]\end{small}%
	$$
	where $\begin{small}\left[\begin{array}{c}\begin{array}{ccc}
	A_{0}-\la B_{0} & \cdots & A_{k-1}-\la B_{k-1}
	\end{array} \\ \hdashline[2pt/2pt] \phantom{\Big|}
	(X-\la Y)\otimes I_{n} \end{array}\right]\end{small}$ is any CORK linearization of $Q(\la).$
\end{deff}
By using Theorem \ref{th:1blockfullrank}, we study in Theorem \ref{linautomatic} the structure of $ \mathcal{L}_{R}(\la)$ as linearization of $R(\la).$

\begin{theo}\label{linautomatic} Let $R(\la)$ be a rational matrix as in \eqref{rational}, and let $\mathcal{L}_{R}(\la)$ be the matrix pencil in Definition \ref{corklinearization}. Let $\Omega\subseteq\C$ be nonempty. If $\mathcal{L}_{R}(\la)$, viewed as a polynomial system matrix with state matrix $(E-\la F)\otimes I_n$, is minimal in $\Omega$ then $\mathcal{L}_{R}(\la)$ is a block full rank linearization of $R(\la)$ in $\Omega$ with state matrix $(E-\la F)\otimes I_n$.
\end{theo}

\begin{proof}
	Set $M(\la):=\left[\begin{array}{ccc}
	A_{0}-\la B_{0} & \cdots & A_{k-1}-\la B_{k-1}
	\end{array}\right],$ $ C(\la):=-[
	a_{1}^{T}\otimes (C_{1}-\la D_{1}) \quad  \cdots \allowbreak \quad  a_{s}^{T}\otimes (C_{s}-\la D_{s})
	],$ $B:=[-b\otimes I_{n} \quad 0],$ $A(\la):=(E-\la F)\otimes I_{n},$ $K_1(\la):= (X-\la Y)\otimes I_{n},$ $N_1(\la):=(f(\la)\otimes I_{n})^{T},$ and $K_2(\la)$ empty. $\mathcal{L}_{R}(\la)$ being minimal in $\Omega$ implies that condition \eqref{minimalitycondition} is satisfied in $\Omega$ since $BN_1(\la)^T=-b\otimes I_n$  because $f_0(\la)=1$. Then, by Theorem \ref{th:1blockfullrank}, $\mathcal{L}_{R}(\la)$ is a linearization of
	$[M(\la)+C(\la)A(\la)^{-1}B](f(\la)\otimes I_{n})=R(\la)$ in $\Omega$ with state matrix $A(\la).$
\end{proof}

\begin{rem} \rm  Theorem \ref{linautomatic} also holds if $f(\la)$ is rational but, in such a case, we need the extra hypothesis of $f(\la)$ being defined in $\Omega$.
\end{rem}

According to Theorem \ref{linautomatic}, we need minimality on $\mathcal{L}_{R}(\la)$ to be a linearization of the rational matrix $R(\la)$. In the following Theorem \ref{suff_cond}, we give sufficient mild conditions for $\mathcal{L}_{R}(\la)$ to be minimal in $\C$ in the case the rational approximants $r_i (\la)$ do not share the same support points and weights.

\begin{theo}\label{suff_cond} Assume that, for $i=1,\ldots,s$, the rational functions $r_{i}(\la)$ in \eqref{aproxmatrix} are represented as in \eqref{quotient} and that this representation is irreducible. Let $\mathcal{L}_{R}(\la)$ be the matrix pencil in Definition \ref{corklinearization}. If the pencils $C_i-\la D_i$ and $E_i-\la F_i$ are regular for $i=1,\ldots,s$ and the following conditions hold
	\begin{itemize}
		\item [\rm(a)] $C_i-\la D_i$ and $E_i-\la F_i$ have no finite eigenvalues in common for $i=1,\ldots,s$, and
		\item [\rm(b)] $E_i-\la F_i$ and $E_j-\la F_j$ with $i\neq j$ have no finite eigenvalues in common  for $i,j=1,\ldots,s$,
	\end{itemize}		
	then $\mathcal{L}_{R}(\la)$, viewed as a polynomial system matrix with state matrix $(E-\la F)\otimes I_n$, is minimal in $\C$.
\end{theo}

\begin{proof}  Assume first that $s=1$. Then notice that $\mathcal{L}_{R}(\la)$ is minimal in $\C$ if the pencil
	$$S(\la):=\left[\begin{array}{c|c}0 &
	a_{1}^{T}\otimes (C_{1}-\la D_{1}) \\ \hline \phantom{\Big|}
	-b_1\otimes I_{n} & (E_1-\la F_1)\otimes I_{n}
	\end{array}\right],$$ considered as a polynomial system matrix with state matrix $(E_1-\la F_1)\otimes I_n$, is minimal in $\C$.
	Since $r_{1}(\la)$ is irreducible, we have, by Proposition \ref{minimality_AAA}, that the submatrix $\left[\begin{array}{c|c}
	-b_1\otimes I_{n} & (E_1-\la F_1)\otimes I_{n}
	\end{array}\right]$ has full row rank for all $\la\in\C$. Then we only have to prove that the submatrix $H(\la):=\left[\begin{array}{c}
	a_{1}^{T}\otimes (C_{1}-\la D_{1}) \\ \hline \phantom{\Big|}
	(E_1-\la F_1)\otimes I_{n}
	\end{array}\right]$ has full column rank for all $\la\in\C$. By contradiction, assume that $H(\la_0)$ has no full column rank for some $\la_0\in\C$. Notice that, in such a case, $\la_0$ must be an eigenvalue of $E_1-\la F_1$ since, otherwise, $H(\la_0)$ would have full column rank. In addition, there exists a nonzero vector $x$ such that $H(\la_0)x=0$. Now we write
	\begin{equation}\label{prod}
	H(\la_0)x= \left[\begin{array}{cc}
	C_{1}-\la_0 D_{1} & 0 \\  \phantom{\Big|}
	0 &  I_{\ell_1 n}
	\end{array}\right] \left[\begin{array}{c}
	a_{1}^{T}\otimes I_n\\ \hline \phantom{\Big|}
	(E_1-\la_0 F_1)\otimes I_{n}
	\end{array}\right]x=0,
	\end{equation}
	and define the vector $\begin{bmatrix}
	y_1 \\y_2
	\end{bmatrix}:=\left[\begin{array}{c}
	a_{1}^{T}\otimes I_n\\ \hline \phantom{\Big|}
	(E_1-\la_0 F_1)\otimes I_{n}
	\end{array}\right]x$, which is nonzero since $x\neq 0$ and the matrix $\left[\begin{array}{c}
	a_{1}^{T}\otimes I_n\\ \hline \phantom{\Big|}
	(E_1-\la_0 F_1)\otimes I_{n}
	\end{array}\right]$ has full column rank by Proposition \ref{minimality_AAA}. Moreover, by \eqref{prod}, we have that $y_2=0$ and, thus, $(C_{1}-\la_0 D_{1} )y_1=0$ with $y_1\neq 0$. Therefore, $\la_0$ is an eigenvalue of $C_{1}-\la D_{1} $, which is a contradiction by condition $\rm(a)$. Finally, if $s>1$ we have to take into account condition $\rm(b)$ and the result follows.
\end{proof}

\begin{rem}\label{comment_omega}\rm It is clear that if we consider the set $\Omega:=\{\la\in\C:E-\la F \text{ is invertible}\},$ then $\mathcal{L}_{R}(\la)$ is minimal in $\Omega$ and, by Theorem \ref{linautomatic}, $\mathcal{L}_{R}(\la)$ is a linearization of $R(\la)$ in $\Omega.$ However, in such a case we do not obtain any information about the poles of $R(\la)$ since they do not belong to $\Omega.$ For this particular choice of the set $\Omega,$ the fact that $\mathcal{L}_{R}(\la)$ is a linearization of $R(\la)$ in $\Omega$ can also be proved by considering $\mathcal{L}_{R}(\la)$ as a block full rank pencil of the form $$\mathcal{L}_{R}(\la):=\left[\begin{array}{c}M(\la) \\ \hdashline[2pt/2pt]K_1(\la)
	\end{array}\right],$$ with $$\begin{small} M(\la):=\left[\begin{array}{c|c}\begin{array}{ccc}
	A_{0}-\la B_{0} & \cdots & A_{k-1}-\la B_{k-1}
	\end{array} & \begin{array}{ccc}
	a_{1}^{T}\otimes (C_{1}-\la D_{1}) & \cdots & a_{s}^{T}\otimes (C_{s}-\la D_{s})
	\end{array}
	\end{array}\right],\end{small}$$ and by applying Remark \ref{rem_emptycasefinite}. For that, write $R(\la)$ as
	\begin{equation*}
	\begin{split}\label{rational2}
	R(\la)=\displaystyle\sum_{i=0}^{k-1}(A_{i}-\la B_{i})(f_{i}(\la)\otimes I_{n}) +\displaystyle\sum_{i=1}^{s}[a_{i}^{T}\otimes (C_{i}-\la D_{i})](R_{i}(\la)\otimes I_{n}),
	\end{split}
	\end{equation*} with $R_{i}(\la):=(E_{i}-\la F_{i})^{-1}b_{i},$ and consider the dual rational basis of $K_1(\la)$
	$$N_1(\la):=[f_{0}(\la)\quad \cdots \quad f_{k-1}(\la) \quad | \quad R_{1}(\la)^{T} \quad \cdots \quad R_{s}(\la)^{T}]\otimes I_n.$$
	Then, $\mathcal{L}_{R}(\la)$ is a linearization of $R(\la)$ in $\Omega$ with empty state matrix. On the other hand, if $\mathcal{L}_{R}(\la)$ (considering the partition with state matrix $(E-\la F)\otimes I_n$) were minimal at those $\la_0\in\C$ such that $E-\la_0 F$ is singular then $\mathcal{L}_{R}(\la)$ would be a linearization of $R(\la)$ in $\C$ with state matrix $(E-\la F)\otimes I_n$. That means that the zeros of $\mathcal{L}_{R}(\la)$ would be the zeros of $R(\la)$, and the zeros of $(E-\la F)\otimes I_n$ would be the poles of $R(\la),$ together with their partial multiplicities. This happens, for instance, under the conditions of Theorem \ref{suff_cond}.
\end{rem}

\begin{rem}\label{notatinfinity} \rm In Remark \ref{comment_omega}, we consider $\mathcal{L}_{R}(\la)$ from two different points of view: as a block full rank pencil, $\left[\begin{array}{c}M(\la) \\ \hdashline[2pt/2pt]K_1(\la)
		\end{array}\right],$ and as a polynomial system matrix with state matrix $(E-\la F)\otimes I_n$. In the former case, $\mathcal{L}_{R}(\la)$ is not in general a linearization at infinity of $R(\la)$ since $\rev_1 K_1(\la)$ has not full row rank at $0$. In particular, \cite[Theorem 5.5]{local} can not be applied and there is not always an integer $g$ such that $\rev_1 \mathcal{L}_{R}(\la)$ is equivalent at $0$ to $\diag (\rev_g R(\la),I_{k(n-1)+\sum_{i=1}^s \ell_{i}n})$. It is not difficult to construct examples where such a $g$ does not exist. In the latter case, $\mathcal{L}_{R}(\la)$ is not a linearization at infinity since $\rev_1 \mathcal{L}_{R}(\la)$ is not minimal at $0$. Both cases are due to the fact that the matrix $\rev_1 (E-\la F)$ has not full row rank at zero since $ F$ is singular.
\end{rem}

\subsection{Low-rank structure}	

Low-rank structures are exploited in \cite{automatic} for constructing smaller linearizations that allow more efficient computations. In particular, a trimmed linearization is constructed if the matrix coefficients $C_{i}-\la D_{i}$ in \eqref{nonlinearmatrix} have low rank. For this purpose, write
\begin{equation}\label{lowrankcoeff}
C_{i}-\la D_{i}=(\widetilde{C}_{i}-\la \widetilde{D}_{i})\widetilde{Z}_{i}^{*},
\end{equation}
with $\widetilde{C}_{i},\widetilde{D}_{i},\widetilde{Z}_{i}\in\C^{n\times k_{i}},$ and $\widetilde{Z}_{i}^{*}\widetilde{Z}_{i}=I_{k_{i}}.$ In several applied problems  this type of structure appears with $k_{i}\ll n$ \cite{nlep, automatic}. By using the expression \eqref{lowrankcoeff} for the matrix coefficients, the matrix $R(\la)$ in \eqref{rational} can be written as:
\begin{equation}
\begin{split}\label{rationallowrank}
R(\la) & =  \displaystyle\sum_{i=0}^{k-1}(A_{i}-\la B_{i})f_{i}(\la) +\displaystyle\sum_{i=1}^{s}(\widetilde{C}_{i}-\la \widetilde{D}_{i})\widetilde{Z}_{i}^{*}a_{i}^{T}(E_{i}-\la F_{i})^{-1}b_{i} \\
& = \displaystyle\sum_{i=0}^{k-1}(A_{i}-\la B_{i})(f_{i}(\la)\otimes I_{n}) +\displaystyle\sum_{i=1}^{s}[a_{i}^{T}\otimes (\widetilde{C}_{i}-\la \widetilde{D}_{i})]((E_{i}-\la F_{i})^{-1}b_{i}\otimes I_{k_{i}})\widetilde{Z}_{i}^{*}.
\end{split}
\end{equation}
Then, the trimmed linearization $\widetilde{\mathcal{L}}_{R}(\la)$ for $R(\la)$ constructed in \cite{automatic} is the following.

\begin{deff}(Trimmed CORK linearization for AAA rational approximation)\label{trimmedcork} Let $R(\la)$ be a rational matrix as in \eqref{rationallowrank}. Consider the matrices
	\begin{equation*}
	\begin{split}
	Z & :=\left[\begin{array}{ccc}
	-\widetilde{Z}_{1}(b_{1}^{*}\otimes I_{k_{1}})&
	\cdots &
	-\widetilde{Z}_{s}(b_{s}^{*}\otimes I_{k_{s}})\end{array}\right],\\   E & :=\diag(E_{1}\otimes I_{k_{1}},\ldots,E_{s}\otimes I_{k_{s}})\text{, and }\\ F & :=\diag(F_{1}\otimes I_{k_{1}},\ldots,F_{s}\otimes I_{k_{s}}).
	\end{split}
	\end{equation*}
	Then a trimmed CORK linearization for $R(\la)$ is
	$$\begin{small} \widetilde{\mathcal{L}}_{R}(\la)=\left[\begin{array}{c|c}\begin{array}{ccc}
	A_{0}-\la B_{0} & \cdots & A_{k-1}-\la B_{k-1}
	\end{array} & \begin{array}{ccc}
	a_{1}^{T}\otimes (\widetilde{C}_{1}-\la \widetilde{D}_{1}) & \cdots & a_{s}^{T}\otimes (\widetilde{C}_{s}-\la \widetilde{D}_{s})
	\end{array} \\ \hdashline[2pt/2pt] \phantom{\Big|}
	(X-\la Y)\otimes I_{n} & 0 \\ \hline \phantom{\Big|}
	\begin{array}{cccccccccccc}  & Z^{*} & & & & & &   0 & &  \end{array} & E-\la F
	\end{array}\right], \end{small}
	$$
	where $\begin{small}\left[\begin{array}{c}\begin{array}{ccc}
	A_{0}-\la B_{0} & \cdots & A_{k-1}-\la B_{k-1}
	\end{array} \\ \hdashline[2pt/2pt] \phantom{\Big|}
	(X-\la Y)\otimes I_{n} \end{array}\right]\end{small}$ is any CORK linearization of $Q(\la).$
\end{deff}

Notice that the linearization $\mathcal{L}_{R}(\la)$ in Definition \ref{corklinearization} has size $(kn+\sum_{i=1}^s \ell_{i}n)\times (kn+\sum_{i=1}^s \ell_{i}n)$ whereas the trimmed pencil $\widetilde{\mathcal{L}}_{R}(\la)$ in Definition \ref{trimmedcork} has size $(kn+\sum_{i=1}^s \ell_{i} k_i)\times (kn+\sum_{i=1}^s \ell_{i} k_i)$ with $k_{i}\ll n$ in several applications.

Analogous to what we did in Theorem \ref{linautomatic}, we study in Theorem \ref{linautomatictrimmed} the structure of $ \widetilde{\mathcal{L}}_{R}(\la)$ as linearization of $R(\la)$. The proof is omitted since it is analogous to that of Theorem \ref{linautomatic}.

\begin{theo}\label{linautomatictrimmed}
	Let $R(\la)$ be a rational matrix as in \eqref{rationallowrank}, and let $\widetilde{\mathcal{L}}_{R}(\la)$ be the matrix pencil in Definition \ref{trimmedcork}. Let $\Omega\subseteq\C$ be nonempty. If $\widetilde{\mathcal{L}}_{R}(\la)$, viewed as a polynomial system matrix with state matrix $E-\la F$, is minimal in $\Omega$ then $ \widetilde{\mathcal{L}}_{R}(\la)$ is a block full rank linearization of $R(\la)$ in $\Omega$ with state matrix $E-\la F$.
\end{theo}
\begin{rem}\rm
		As we discussed in Remark \ref{notatinfinity} for the matrix pencil $\mathcal{L}_{R}(\la),$ the trimmed CORK linearization
		$ \widetilde{\mathcal{L}}_{R}(\la)$ is not in general a linearization at infinity of $R(\la)$ either. The reason is that, in this case, the matrix $ F$ is also singular and $\rev_1 (E-\la F)$ has not full row rank at zero.
\end{rem}

\section{Conclusions }\label{sect:con}

Combining the theory in \cite{local} with a nontrivial extension of the structure of the strong block minimal bases linearizations introduced in \cite{strong}, we have constructed a new wide family of local linearizations of rational matrices that generalizes and includes most of the linearizations for rational matrices appearing in the literature. The linearizations in this family are called block full rank linearizations. Depending on the satisfied minimality conditions, a pencil in this family can be a linearization in a set of finite points and/or at infinity. If the minimality conditions are satisfied in the whole underlying field $\mathbb{F}$ and at infinity, simultaneously, then we can recover from block full rank linearizations the complete pole and zero information, finite and at infinity, of rational matrices. Linearizations at infinity are defined using the notion of grade and,  to determine the grade of block full rank linearizations at infinity, we use the notion of degree of rational matrices.

As an application, we use block full rank linearizations to study the structure of the linearizations developed in \cite{automatic} for solving rational eigenvalue problems coming from rational approximations of nonlinear eigenvalue problems. We provide sufficient mild conditions under which the pencils in \cite{automatic} are linearizations in $\mathbb{C}$.

\section*{Funding}

The authors F. M. Dopico and M. C. Quintana are supported by ``Ministerio de Econom\'ia, Industria y Competitividad (MINECO)" of Spain and ``Fondo
Europeo de Desarrollo Regional (FEDER)" of EU through grants MTM2015-65798-P and MTM2017-90682-REDT
and by the ``Proyecto financiado por la Agencia Estatal de Investigaci\'on (PID2019-106362GB-I00 / AEI / 10.13039/501100011033)".
The research of M. C. Quintana is also funded by the ``contrato predoctoral" BES-2016-076744 of MINECO. The second author, S. Marcaida, is supported by ``Ministerio de Econom\'ia, Industria y Competitividad (MINECO)'' of Spain and ``Fondo Europeo de Desarrollo Regional (FEDER)'' of EU through grant MTM2017-83624-P, and by UPV/EHU through grant PPGA20/10. This work was developed while P. Van Dooren held a ``Chair of Excellence UC3M - Banco de Santander'' at Universidad Carlos III de Madrid in the academic years 2017-2018 and 2019-2020.


\begin{thebibliography}{99}

		\bibitem{AlBe16} R.~Alam, N.~Behera,
\textit{Linearizations for rational matrix functions and Rosenbrock
	system polynomials}, SIAM J. Matrix Anal. Appl. 37(1) (2016) 354--380.

\bibitem{AlBe16-2}
R. Alam, N. Behera, \textit{Recovery of eigenvectors of rational matrix functions from Fiedler-like linearizations}, Linear Algebra Appl. 510 (2016) 373--394.

\bibitem{AlBe18} R. Alam, N. Behera, \textit{Generalized Fiedler pencils for rational matrix functions}, SIAM J. Matrix Anal. Appl. 39(2) (2018) 587--610.

\bibitem{strong} A.~Amparan, F.~M. Dopico, S.~Marcaida, I.~Zaballa, \textit{Strong linearizations of rational matrices}, SIAM J. Matrix Anal. Appl. 39(4) (2018) 1670--1700.

\bibitem{minimal} A.~Amparan, F.~M. Dopico, S.~Marcaida, I.~Zaballa, \textit{On minimal bases and indices of rational matrices and their linearizations}, submitted. Available as arXiv:1912.12293.

\bibitem{AmMaZa15} A.~Amparan, S.~Marcaida, I.~Zaballa,
\textit{Finite and infinite structures of rational matrices: a local approach}, Electron. J. Linear Algebra 30 (2015) 196--226.

\bibitem{DasAl19laa} R. K. Das, R. Alam, \textit{Recovery of minimal bases and minimal indices of rational matrices from Fiedler-like pencils}, Linear Algebra Appl. 566 (2019) 34--60.


\bibitem{DaAl19_2} R. K. Das, R. Alam,
\textit{Affine spaces of strong linearizations for rational matrices and the recovery of eigenvectors and minimal indices}, Linear Algebra Appl. 569 (2019) 335--368.

\bibitem{DaAl20} R. K. Das, R. Alam, \textit{Structured strong linearizations of structured rational matrices},  arXiv:2008.00427v1

\bibitem{DoMaQu19}
F. M. Dopico, S. Marcaida, M. C. Quintana,
\textit{Strong linearizations of rational matrices with polynomial part expressed in an orthogonal basis}, Linear Algebra Appl. 570 (2019) 1--45.

\bibitem{local}
F.~M.~Dopico, S.~Marcaida, M.~C.~Quintana, P.~Van Dooren,
\newblock\textit{Local linearizations of rational matrices with application to rational approximations of nonlinear eigenvalue problems}, Linear Algebra Appl. 604 (2020) 441--475.


\bibitem{BKL}F.~M. Dopico, P.~W. Lawrence, J.~Pérez, P.~Van Dooren, \textit{Block Kronecker linearizations of matrix polynomials and their backward errors}, Numer. Math. 140 (2018) 373--426.	

\bibitem{dopico-pizarro}  F. M. Dopico, J. Gonz\'alez-Pizarro, \textit{A compact rational Krylov method for large-scale rational eigenvalue problems}, Numer. Linear Algebra Appl. 26  (2019) e2214 (26pp).

\bibitem{Saad} M. El-Guide, A. Miedlar,  Y. Saad, \textit{A rational approximation method for solving acoustic nonlinear eigenvalue problems}, Eng. Anal. Bound. Elem. 111 (2020) 44-54.

\bibitem{forney} G.~D. Forney, Jr., \textit{Minimal bases of rational vector spaces, with applications to multivariable linear systems,} SIAM J. Control 13(3) (1975) 493--520.

\bibitem{guttel-tisseur-2017} S. Güttel, F. Tisseur, \textit{The nonlinear eigenvalue problem}, Acta Numer. 26 (2017) 1--94.

\bibitem{nlep} S. Güttel, R. Van Beeumen, K. Meerbergen, W. Michiels, \textit{NLEIGS: A class of fully rational Krylov methods for nonlinear eigenvalue problems}, SIAM J. Sci. Comput. 36(6) (2014) A2842–A2864.

\bibitem{Kailath} T.~Kailath, \textit{Linear Systems}, Prentice Hall, New Jersey, 1980.

\bibitem{automatic}
P. Lietaert, J. Pérez, B. Vandereycken, K. Meerbergen, \textit{Automatic rational approximation and linearization of nonlinear eigenvalue problems,} submitted. Available as arXiv:1801.08622v2.

\bibitem{lu-huang-bai-su-2015}
D. Lu, X. Huang, Z. Bai, Y. Su, \textit{A Padé approximate linearization algorithm for solving the quadratic eigenvalue problem with low-rank damping,} Int. J. Numer. Meth. Engng. 103 (2015) 840–858.

\bibitem{McMi2} B. McMillan, \textit{Introduction to formal realizability theory II,} Bell System Tech. J. 31 (1952) 541--600.

\bibitem{mehrmanvoss2004} V. Mehrmann, H. Voss, \textit{Nonlinear eigenvalue problems: A challenge for modern eigenvalue methods}, GAMM--Mitt. 27 (2004) 121--152.

\bibitem{moler-stewart} C. B. Moler, G. W. Stewart, \textit{An algorithm for generalized matrix eigenvalue problems}, SIAM J. Numer. Anal. 10 (1973) 241–256.

\bibitem{AAA} Y. Nakatsukasa, O. Sète, L. N. Trefethen, \textit{The AAA algorithm for rational approximation}, SIAM J. Sci. Comput. 40(3)  (2018)  A1494--A1522.

\bibitem{Rosen} H. H. Rosenbrock,
\textit{State-space and Multivariable Theory,} Thomas Nelson and Sons, London, 1970.

\bibitem{su-bai-2011} Y. Su, Z. Bai, \textit{Solving rational eigenvalue problems via linearization}, SIAM J. Matrix Anal. Appl. 32 (1) (2011) 201--216.

\bibitem{van-beeumen-et-al-2018} R. Van Beeumen, O. Marques, E. G. Ng, C. Yang, Z. Bai, L. Ge, O. Kononenko, Z. Li, C.-K. Ng, L. Xiao, \textit{Computing resonant modes of accelerator cavities by solving nonlinear eigenvalue problems via rational approximation}, J. Comput. Phys. 374 (2018) 1031--1043.

\bibitem{cork} R. Van Beeumen, K. Meerbergen, W. Michiels, \textit{Compact rational Krylov methods for nonlinear eigenvalue problems}, SIAM J. Matrix Anal. Appl. 36(2) (2015) 820--838.	


\bibitem{vd1979} P. Van Dooren, \textit{The computation of Kronecker's canonical form of a singular pencil}, Linear Algebra Appl. 27 (1979) 103--140.


\bibitem{vd1981} P. Van Dooren, \textit{The generalized eigenstructure problem in linear system theory}, IEEE Trans. Automat. Contr. 26 (1) (1981) 111--129.

\bibitem{vandooren-laurent-1979} P. Van Dooren, P. Dewilde, J. Vandewalle, \textit{On the determination of the Smith-McMillan form of a rational matrix from its Laurent expansion}, IEEE Trans. Circuit Syst. 26(3) (1979) 180--189.

\bibitem{Vard} A. I. G. Vardulakis, \textit{Linear Multivariable Control,} John Wiley and Sons, New York, 1991.

\bibitem{VeDoka79} G. Verghese, P. Van Dooren, T. Kailath,
\textit{Properties of the system  matrix of a generalized state-space system}, Int. J. Control 30(2) (1979) 235--243.


\end{thebibliography}
\end{document}